\documentclass[reqno,final]{amsart}

\usepackage{amsmath,amsfonts,amsthm,amssymb,color}
\usepackage[T1]{fontenc}
\usepackage{pdfsync}
\usepackage{csquotes}
\usepackage{graphicx}
\usepackage{pstricks}
\usepackage{lmodern}

\newcommand{\id}{\mbox{Id}}

\newcommand{\1}{{\bf 1}}

\newcommand{\bu}{\mathbf{U}}
\newcommand{\bv}{\mathbf{V}}

\newcommand{\bof}{\mathbf{F}}

\newcommand{\C}{\mathbb C}

\newcommand{\R}{\mathbb R}

\newcommand{\U}{\mathbb U}

\newcommand{\ca}{\mathcal A}

\newcommand{\cac}{\mathcal C}

\newcommand{\ce}{\mathcal E}

\newcommand{\ch}{\mathcal H}

\newcommand{\cp}{\mathcal P}

\newcommand{\al}{\alpha}
\newcommand{\der}{\delta}

\newcommand{\vp}{\varphi}

\newtheorem{theorem}{Theorem}[section]

\newtheorem{corollary}[theorem]{Corollary}

\newtheorem{definition}[theorem]{Definition}

\newtheorem{lemma}[theorem]{Lemma}

\newtheorem{proposition}[theorem]{Proposition}

\theoremstyle{remark}
\newtheorem{remark}[theorem]{Remark}

\date{\today}

\begin{document}

\makeatletter
\def\@settitle{\begin{center}%
  \baselineskip14\p@\relax
    \normalfont\LARGE
\@title
  \end{center}%
}
\makeatother

\title[Skorohod and rough integration with respect to the NC-fBm]{Skorohod and rough integration with respect to the non-commutative fractional Brownian motion}

\author{Aur\'elien Deya}
\address[A. Deya]{Institut Elie Cartan, University of Lorraine
B.P. 239, 54506 Vandoeuvre-l\`es-Nancy, Cedex
France}
\email{aurelien.deya@univ-lorraine.fr}

\author{Ren\'e Schott}
\address[R. Schott]{Institut Elie Cartan, University of Lorraine
B.P. 239, 54506 Vandoeuvre-l\`es-Nancy, Cedex
France}
\email{rene.schott@univ-lorraine.fr}



\begin{abstract}
We pursue our investigations, initiated in \cite{deya-schott-3}, about stochastic integration with respect to the non-commutative fractional Brownian motion (NC-fBm). Our main objective in this paper is to compare the pathwise constructions of \cite{deya-schott-3} with a Skorohod-type interpretation of the integral.

\smallskip

As a first step, we provide details on the basic tools and properties associated with non-commutative Malliavin calculus, by mimicking the presentation of Nualart's celebrated treatise \cite{nualart-book}. Then we check that, just as in the classical (commutative) situation, Skorohod integration can indeed be considered in the presence of the NC-fBm, at least for a Hurst index $H>\frac14$.

\smallskip

This finally puts us in a position to state and prove the desired comparison result, which can be regarded as an It{\^o}-Stratonovich correction formula for the NC-fBm.

\end{abstract}

\maketitle

\section{Introduction}

This study can be seen as the continuation of our previous paper \cite{deya-schott-3}. The two works share the same general objective, namely to investigate integration issues related to the non-commutative fractional Brownian motion (NC-fBm in the sequel). Let us first recall that this topic lies at the intersection of three important fields (we will of course go back in detail to each of these points in the sequel):

\smallskip

\noindent
$\bullet$ the theory of fractional processes, which, as far as modeling is concerned, aims at more flexibility than the usual Brownian noises;

\smallskip

\noindent
$\bullet$ the theory of stochastic integration, or how to overcome, in a differential context, the difficulties steming from the irregularity of the most interesting stochastic processes;

\smallskip

\noindent
$\bullet$ the theory of non-commutative processes, that is the analysis of processes with values in a non-commutative probability space.
 
\

Thus, through the subsequent study, we intend to bring a contribution - even a modest one - to each of these general areas. Our objective is also to provide new elements of comparison, whether similarities or differences, between the classical and the non-commutative probability settings.

\

The object at the center of the study is the NC-fBm, a process which first occurs in \cite{nourdin-taqqu} within a central-limit-theorem result. As its name suggests, the NC-fBm is the counterpart, in the non-commutative probability framework, of the classical fractional Brownian motion. This analogy can at least be justified along two (correlated) directions:

\smallskip

\noindent
$(i)$ First, the NC-fBm is a family of semicircular processes, the non-commutative analog of the Gaussian processes. Just as Gaussian processes (in the classical world), semicircular processes play a central role in the non-commutative probability theory, and they are also known to be characterized by their mean and covariance functions. The mean and covariance functions of the NC-fBm are precisely those of the classical fBm, as can be seen in the subsequent Definition \ref{defi:NC-fbm}. 

\smallskip

\noindent
$(ii)$ Secondly, let us recall that a fundamental feature of non-commutative probability theory is its close links with random matrix theory. In his seminal paper \cite{voiculescu}, Voiculescu showed in particular that the $d$-dimensional Hermitian Brownian motion (i.e. the family of $(d\times d)$-Hermitian matrices with upper-diagonal entries given by independent complex Brownian motions) converges, in the spectral sense and as $d\to \infty$, to the so-called free Brownian motion (i.e. the centered semicircular process with covariance given by the standard Brownian covariance). It turns out that this convergence property can be extended to the fractional situation: starting from a Hermitian fBm, the new limit then precisely corresponds to the NC-fBm (see \cite[Proposition 3.6]{deya-hfbm} for more details).

\

The question of stochastic integration with respect to a non-commutative process was first raised in the breakthrough paper \cite{biane-speicher} by Biane and Speicher, with the construction of an It{\^o}-type integral with respect to the free Brownian motion. These considerations (or at least a part of them) were then extended to the NC-fBm in \cite{deya-schott-3} using the so-called \emph{rough paths}, or \emph{pathwise}, approach developed in \cite{deya-schott} (we will report on those results in Section \ref{subsec:pathwise-integr} below).

\smallskip

Beyond the interest for a \enquote{reasonable definition of the stochastic integral}, the study of integration issues often sheds new light on the properties and general behaviour of the process under consideration. For instance, studying integration with respect to the free Brownian motion points out the central role of the free independence property satisfied by its disjoint increments. In the (non-commutative) fractional situation, where free independence is no longer available, such an analysis tells us in particular that the infinitesimal variations of the process can be easily controlled as long as the regularity coefficient $H$ of the process is strictly larger than $\frac12$, but then this control requires sophisticated \enquote{second-order} tools as soon as $H<\frac12$ (note that the intermediate case $H=\frac12$ corresponds to the free Brownian motion). When $H\leq \frac14$, the process even happens to be \enquote{locally too non-commutative} to allow a suitable control of infinitesimal increments, and accordingly the exhibition of a stochastic integral (see \cite[Remark 2.7 and Proposition 2.11]{deya-schott-3}). In each of these situations, the construction also emphasizes the fundamental role of the semicircular property, with an extensive use of the non-commutative Wick formula throughout the procedure (formula \eqref{form-wick-original} below).

\

In this paper, we would like to go even further into this semicircular analysis, by considering another general approach to stochastic integration, namely the \emph{Malliavin calculus} approach, leading to the so-called \emph{Skorohod integral}. 

\smallskip

A first part of the study (Section \ref{sec:skoro-int}) will thus be devoted to the presentation of Malliavin calculus in the non-commutative setting, for a given general semicircular process. The first developments on \enquote{non-commutative Malliavin calculus} can be again traced back to the aforementioned paper \cite{biane-speicher} by Biane and Speicher. Our below presentation will slightly differ from theirs (see Remarks \ref{rk:differ-1} and \ref{rk:differ-2} below), and in fact, our objective will be to stay as close as possible to the classical presentation of Malliavin calculus (i.e., in the commutative setting), especially the presentation in \cite{nualart-book}. We hope that this similarity can make the introduction of these tools easily accessible to non \enquote{NC experts}. 

\smallskip

The idea then will be to illustrate this approach through the NC-fBm example, so as to define the \emph{Skorohod integral with respect to the NC-fBm}, at least for a non-trivial class of integrands (see Proposition \ref{prop:skoro-integr}).

\

In the classical probability setting, Skorohod integration is often considered as the natural extension of It{\^o} integration, owing to its very \enquote{stochastic} nature, while the pathwise approaches are rather seen as extensions of the Stratonovich interpretation (see Remark \ref{rk:strato-int} for more details). Following this idea, any comparison result between the Skorohod and the pathwise integral is classically referred to as an \emph{It{\^o}-Stratonovich (correction) formula}. When dealing with a one-dimensional fBm $B$ of Hurst index $H>\frac14$, the following comparison formula can for instance be found in \cite{cheridito-nualart} (see also \cite[Section 5.2.3]{nualart-book}): for any $f\in \cac^\infty(\R;\R)$ such that $f$ and its derivatives are of polynomial growth,
\begin{equation}\label{ito-strato-commuta}
\int_s^t f(B_u) \, dB_u=\int_s^t f(B_u)\, \delta B_u+H\int_s^t f'(B_u)\, u^{2H-1} \, du \, ,
\end{equation}
where the integral in the left-hand side is understood in the pathwise \enquote{Stratonovich} sense, while the integral in the right-hand side is understood in the Skorohod \enquote{It{\^o}} sense. Observe that when $H=\frac12$, that is when $B$ is a standard Brownian motion, formula \eqref{ito-strato-commuta} reduces to the standard It{\^o}-Stratonovich formula.

\smallskip

Naturally, these correction formulas are closely related to the so-called \enquote{It{\^o} formulas}, that account for the differential rule satisfied by It{\^o} or Skorohod integral. Using \eqref{ito-strato-commuta}, together with some standard properties of the pathwise integral, we get for instance (see \cite[Theorem 1]{alos-mazet-nualart}) that  
$$f(B_t)-f(B_s)=\int_s^t f'(B_u) \, \delta B_u+H\int_s^t f''(B_u)\, u^{2H-1} \, du \, ,$$
provided $H>\frac14$ and $f\in \cac^\infty(\R;\R)$, with derivatives of polynomial growth. A multidimensional version of this result has also been established in \cite{hu-jolis-tindel} (note that, considering the above interpretation $(ii)$ of the NC-fBm, the multidimensional setting is clearly closer to the framework of the paper): given $B=(B^1,\ldots,B^d)$ a $d$-dimensional fBm of Hurst index $H>\frac14$, one has, for any $f\in \cac^\infty(\R^d;\R)$ such that $f$ and its partial derivatives are of polynomial growth,
$$f(B_t)-f(B_s)=\int_s^t \langle \nabla f(B_u) , \delta B_u\rangle+H\int_s^t \Delta f(B_u)\, u^{2H-1} \, du \, ,$$
with the usual notation $\nabla f(x):=(\partial_{x_1}f(x),\ldots, \partial_{x_d}f(x))$ and $\Delta f (x):=\sum_{i=1}^d \partial^2_{x_i}f(x)$. Again, when $H=\frac12$, we immediately recover the classical It{\^o} formula for the standard Brownian motion.

\

As we will see in the sequel, such a comparison between Skorohod and pathwise integrals is still possible in the NC probability setting (when working with a NC-fBm), using a specific \enquote{non-commutative refinement} of the correction term. This is the topic of Theorem \ref{main-theo} below, which can be considered as the main result of the paper.

\smallskip

Before we can state and prove this formula, we will of course need to briefly remind the reader with some preliminary existence results about the pathwise integral with respect to the NC-fBm, as they are displayed in \cite{deya-schott-3} (see Section \ref{subsec:pathwise-integr} below). The key object behind these results is the so-called \enquote{L{\'e}vy area} term $\mathbb{X}^{2}$, corresponding to the non-commutative counterpart of the genuine L{\'e}vy area of rough paths theory, and providing the suitable correction to the usual Riemann sum (see Proposition \ref{prop:rough-int}). As a natural consequence of this central role, the desired comparison between Skorohod and pathwise integrals will first require a comparison result at the level of the L{\'e}vy area term: this will be the purpose of Proposition \ref{prop:rough-case}, our main technical result in this analysis. The strategy can here be compared with some of the arguments used in the recent study \cite{cass-lim} by Cass and Lim towards a general It{\^o}-Stratonovich formula for the solutions of rough differential systems (in the classical commutative setting).

\

\

The paper is organized in accordance with the previous description. In Section \ref{sec:skoro-int}, we first recall some basics about the non-commutative probability setting (Sections \ref{subsec:nc-setting} and \ref{subsec:semicirc-proc}), and then go on with a slightly reshaped presentation (compared to the one in \cite{biane-speicher}) of the non-commutative Malliavin calculus associated with a general semicircular process. From Section \ref{sec:int-nc-fbm}, we will restrict our attention to the case of the NC-fBm (Definition \ref{defi:NC-fbm}). We will first check that the conditions ensuring the existence of the Skorohod integral are indeed satisfied in this situation (Section \ref{subsec:skoro-fbm}), and also briefly recall some previous results about pathwise integration (Section \ref{subsec:pathwise-integr}). This will naturally settle the stage for our main result, stated in Section \ref{sec:ito-strato-correc}, that is the It{\^o}-Stratonovich correction formula (Theorem \ref{main-theo}). Finally, Appendix \ref{sec:proof-approx-sum-levy} is devoted to the proof of our main technical property about the \enquote{local approximation} of the L{\'e}vy-area term (Proposition \ref{prop:rough-case}).

\

\emph{Although the framework and the objects of this study are quite specific, we have tried to make their presentation as self-contained as possible, and so (hopefully) accessible to a large audience.} 

\

\textbf{Acknowledgements.} We are deeply grateful to an anonymous reviewer for his/her very careful reading and his/her highly detailed report. This report has led to several significant clarifications in our study, and also entailed numerous improvements in the presentation of our results.

\

\section{Skorohod integration with respect to a semicircular process}\label{sec:skoro-int}

Before we can turn to the presentation of the non-commutative Malliavin calculus (and its associated Skorohod integral), we first need to recall a few basics about the general framework of our study: the non-commutative probability theory.  

\smallskip

\emph{Note that, for the sake of conciness, we will often (not to say always) use the shortcut notation NC for \enquote{non-commutative} in the sequel.}

\subsection{NC probability spaces: setting and notations}\label{subsec:nc-setting}

\

\smallskip

\begin{definition}\label{defi:nc-proba-space}
A NC probability space consists of a pair $(\ca,\vp)$, where:

\smallskip

\noindent
$(i)$ $\ca$ is a unital algebra over $\C$, equipped with an antilinear $\ast$-operation $X\mapsto X^\ast$ satisfying $(X^\ast)^\ast=X$ and $(XY)^\ast=Y^\ast X^\ast$ for all $X,Y\in \ca$. Also, there must exist a norm $\| .\| : \ca \to [0,\infty[$ which makes $\ca$ a Banach space, and such that $\| XY\|\leq \|X\| \|Y\|$ and $\|X^\ast X\|=\|X\|^2$, for all $X,Y\in \ca$.

\smallskip

\noindent
$(ii)$ $\vp:\ca \to \C$ (the \enquote{trace}, or \enquote{expectaction}) is a linear functional on $\ca$ satisfying $\vp(1)=1$, $\vp(XY)=\vp(YX)$, $\vp(X^\ast X)\geq 0$ for all $X,Y\in \ca$, and $\vp(X^\ast X)=0 \Leftrightarrow X=0$. 

\smallskip

Once endowed with a NC probability space, we call any $X\in \ca$ a NC random variable, and accordingly any path $X:[0,T] \to \ca,\, t\mapsto X_t$ is a NC process. 
\end{definition}

\smallskip

\begin{remark}
This \enquote{NC probability space} terminology is in fact a simplification with respect to the more specific vocabulary presented in \cite{nica-speicher}. Following the latter reference, the above structure should rather be referred to as a \emph{tracial $C^\star$-probability space with faithful trace}, and is a more restrictive setting than the general non-commutative probability framework (as defined in \cite[Definition 1.1]{nica-speicher}).
\end{remark}

\smallskip

\begin{remark}
The \enquote{random variable} terminology can be further justified through the existence of some underlying probability law having the same moments as $X\in \ca$ (where the moments for $X$ are understood in the sense of the trace). Again, an exhaustive presentation of these features can be found in \cite{nica-speicher}. 
\end{remark}

\smallskip

Using the above points $(i)$-$(ii)$, it is easy to see that the map 
$$(X,Y) \mapsto \langle X,Y\rangle_{L^2(\vp)} :=\vp\big( X Y^\ast\big)$$
defines a (complex) inner product in $\ca$. As usual, we will denote the completion of $\ca$ with respect to $\langle .,.\rangle_{L^2(\vp)}$ as $L^2(\vp)$, and somehow see this space as the NC analog of the classical $L^2(\Omega)$-space (note however that $\ca$ may be strictly contained in $L^2(\vp)$). Just as in the classical commutative case, the $L^2(\vp)$-norm will be the reference topology in the subsequent developments on NC Malliavin calculus.

\smallskip

Besides, it can be shown (see \cite[Proposition 3.17]{nica-speicher}) that the norm $\|.\|$ in point $(i)$ is necessarily linked to $\vp$ through the relation
\begin{equation}\label{trace-norm-2}
\|X\|=\lim_{r\to \infty} \vp\big(\big(X X^\ast\big)^r\big)^{\frac{1}{2r}} \, .
\end{equation}
Based on this fundamental property, we will sometimes write the norm $\|.\|$ as $\|.\|_{L^\infty(\vp)}$, so as to make a clear distinction with the $L^2(\vp)$-topology.

\

When studying differential properties in a NC structure, the tensor product of the space is expected to play a central role. Here, starting from a NC probability space $(\ca,\vp)$, we can also endow the (algebraic) tensor product $\ca \otimes \ca$ with a tracial structure. To this end, we first define the product, resp. the $\ast$-operation, through the bilinear, resp. linear, extension of the formulas
$$(F_1 \otimes F_2)\cdot (G_1\otimes G_2):=(F_1 G_1)\otimes (F_2 G_2) \quad , \quad \text{resp.} \quad  (F_1 \otimes F_2)^\ast:=F_1^\ast \otimes F_2^\ast \, ,$$
and then consider the trace $\vp\times \vp$ defined as the linear extension of
\begin{equation}\label{product-trace}
\big( \vp\times \vp\big)\big( F_1\otimes F_2\big):=\vp(F_1)\vp(F_2) \, .
\end{equation}
The fact that the so-defined form $\vp\times \vp$ is indeed a trace  on $\ca \otimes \ca$ (in the sense of Definition \ref{defi:nc-proba-space}, item $(ii)$) is not immediate, but it can be shown for instance through the use of orthonormal elements of $(\ca,\vp)$. Note that the resulting pair $(\ca \otimes \ca,\vp \times \vp)$ is not exactly a NC probability space (in the sense of Definition \ref{defi:nc-proba-space}), due its lack of completeness.

\smallskip

Of course, the above construction of a tracial structure can then be extended to any tensor product $\ca^{\otimes k}$, $k\geq 3$.

\smallskip

On the other hand, the following operation, that we will occasionally use in the sequel, is specific to $\ca\otimes \ca$: namely, we set, for all $F_1,F_2,G\in \ca$,
\begin{equation}\label{sharp-notation}
\big( F_1 \otimes F_2) \sharp G:=F_1GF_2 \, 
\end{equation}
and then linearly extend this definition to all $\bof\in \ca\otimes \ca$ and $G\in \ca$. 

\smallskip

With this notation in hand, one can check that for any polynomial function $P(x):=\sum_{k=0}^d a_k \, x^k$,
\begin{equation}\label{deriv-rule-tensor-product}
\lim_{\varepsilon \to 0} \frac{1}{\varepsilon} \big\{P(X+\varepsilon Y)-P(X) \big\}=\partial P(X)\sharp Y \, ,
\end{equation}
where the \emph{tensor derivative} $\partial P(X)$ is the element in $\ca \otimes \ca$ defined as
\begin{equation}\label{partial-op}
\partial P(X):=\sum_{k=1}^d a_k\sum_{i=0}^{k-1} X^i \otimes X^{k-1-i} \, .
\end{equation}
Extending such a property at second order (which we shall use in the sequel) naturally leads us to the consideration of the second-order tensor derivative: namely, we set
\begin{equation}\label{d-deux-p}
\partial^2  P(X):=\sum_{k=2}^d a_k \sum_{\substack{i,j \geq 0\\i+j \leq k-2}} X^i \otimes X^j \otimes X^{k-2-i-j} \ \in \ca \otimes \ca \otimes \ca \, .
\end{equation}

\subsection{Semicircular processes}\label{subsec:semicirc-proc}

\

\smallskip

Let us first recall that for every $m\geq 1$, a \emph{pairing} of $\{1,\ldots,2m\}$ is a partition of the latter set into $m$ disjoint pairs. 

\smallskip

Now, in NC probability theory, special attention is paid to the \emph{non-crossing pairings}: those are the pairings $\pi$ for which there are no elements $\{p_1,q_1\},\{p_2,q_2\}\in \pi$ with $p_1 < p_2 <q_1 <q_2$. In the sequel, we will denote by $NC_2(r)$ the set of the non-crossing pairings of $\{1,\ldots,r\}$. This set appears in particular in the definition of the following central family of (NC) random variables:

\begin{definition}\label{def:semicircular-family}
Given a NC probability space $(\ca,\vp)$, a  (centered) semicircular family is a collection $\{X_i\}_{i \in I}$ of self-adjoint elements in $\ca$ (i.e. $X_i^\ast=X_i$) such that, for every even integer $r\geq $1 and all $i_1,\ldots,i_r \in I$, one has the identity
\begin{equation}\label{form-wick-original}
\vp\big( X_{i_1}\cdots X_{i_r}\big)=\sum_{\pi \in NC_2(r)} \prod_{(p,q)\in \pi} \vp\big(X_{i_p} X_{i_q}\big) \ ,
\end{equation}
and $\vp\big( X_{i_1}\cdots X_{i_r}\big)=0$ whenever $r$ is an odd integer.
\end{definition}

Semicircular processes are nothing but the NC analog of the Gaussian processes, as can be seen from the so-called \emph{free central limit theorem} (see e.g. \cite[Theorem 8.7]{nica-speicher}), where (classical) independence of random variables is replaced with the fundamental free independence property. It turns out that we will never appeal to this freeness property in the subsequent considerations, and therefore we refrain from elaborating on it.  

\smallskip

In the sequel, we will also be led to use the following convenient notation: for all $X_1,\ldots,X_{2m} \in \ca$ and for every pairing $\pi$ of $\{1,\ldots,2m\}$, 
\begin{equation*}
\vp_\pi\big( X_1,\ldots,X_{2m}\big):=\prod_{(p,q)\in \pi} \vp\big( X_{p}X_{q}\big) \, ,
\end{equation*}
which allows us to rewrite the NC Wick formula \eqref{form-wick-original} as
\begin{equation}\label{form-wick}
\vp\big( X_{i_1}\cdots X_{i_r}\big)=\sum_{\pi \in NC_2(r)} \vp_\pi\big( X_{i_1},\ldots,X_{i_r}\big) \ .
\end{equation}

Let us also label the following immediate consequence of \eqref{form-wick-original} for further reference:
\begin{proposition}\label{prop:wick}
Given a (centered) semicicular vector $(X_1,\ldots,X_r)$ in a NC probability space $(\ca,\vp)$, one has, for every $i=1,\ldots,r$,
\begin{align*}
\vp\big(X_1 \cdots X_r\big)=&\sum_{j<i} \vp\big(X_i X_j \big) \vp(X_1\cdots X_{j-1} X_{i+1} \cdots X_r \big)\vp\big(X_{j+1}\cdots X_{i-1}\big)\\
&+\sum_{j>i} \vp\big(X_i X_j) \vp\big(X_1\cdots X_{i-1}X_{j+1}\cdots X_r\big)\vp\big(X_{i+1}\cdots X_{j-1}\big) \, .
\end{align*}
\end{proposition}

\subsection{NC Malliavin calculus}

\

\smallskip

Let us observe first that in \cite{biane-speicher} (or in \cite{knps}), the introduction of NC Malliavin calculus strongly leans on the possible representation, in law, of any semicircular process as a path with values in the so-called \emph{full Fock space}. In particular, the basic derivative and divergence operators are therein defined as maps acting on this full Fock space (see \cite[Definitions 5.1.1 and 5.1.2]{biane-speicher} or \cite[Definition 3.3]{knps}).

\smallskip

For the reader's convenience, and also to make the analogy with classical Malliavin calculus even more obvious, we have here preferred to reformulate the whole presentation \emph{independently of any particular representation of the process}.

\smallskip

\textit{Therefore, from now on and for the rest of Section \ref{sec:skoro-int}, we fix a general NC probability space $(\ca,\vp)$, as well as a generic centered semicircular process $\{X_t, \, t\in [0,T]\}$ on $\ca$, for some finite time horizon $T> 0$.}

\smallskip

\textit{Let us also specify that we will essentially restrict our attention to the unital subalgebra $\ca_X$ of $\ca$ generated by $\{X_t, \, t\in [0,T]\}$.}

\

Just as in the classical Malliavin calculus theory, we will focus on the Hilbert space $\ch$ associated with $X$, that is $\ch$ is the completion of the space of elementary functions with respect to the product
\begin{equation}\label{defi-ch}
\langle \1_{[0,s]},\1_{[0,t]}\rangle_{\ch}=\vp \big( X_s X_t\big) \, .
\end{equation}

\smallskip

For more clarity, we will henceforth denote the space of elementary functions by
$$\ce([0,T];\R):=\Big\{h=\sum_{i=1}^r \al_i \1_{[a_i,b_i)} \, , \, r\geq 1 \, , \, \al_i \in \R \, , \, 0\leq a_i<b_i\leq T\Big\} \, ,$$
and then, for any algebra $E$ (whether $\ca_X$, $\ca_X\otimes \ca_X$,...), we set
$$\ce([0,T];E):=\Big\{t\mapsto \sum_{i=1}^r x_i h_i(t) \, , \, r \geq 1 \, , \, x_i\in E \, , \, h_i\in \ce([0,T];\R) \Big\} \, ,$$
$$\ce([0,T]^2;E):=\Big\{(s,t)\mapsto \sum_{i=1}^r x_i h_i(s)k_i(t) \, , \, r\geq 1 \, , \, x_i\in E \, , \, h_i,k_i\in \ce([0,T];\R) \Big\} \, .$$
For any $h=\sum_{i=1}^r \al_i \1_{[a_i,b_i)} \in \ce([0,T];\R)$, we set as usual 
$$X(h):=\sum_{i=1}^r \al_i  (X_{b_i}-X_{a_i})\, .$$
Finally, we extend the product $\langle .,.\rangle_{\ch}$ along the following natural rules: for all $x,y\in E$ and $h_1,h_2\in \ce([0,T];\R)$,
$$\langle x h_1,h_2\rangle_{\ch}=\langle h_1,xh_2\rangle_{\ch}:=x \langle h_1,h_2\rangle_{\ch} \, \in E \, ,$$
$$\langle x h_1;y h_2\rangle_{\ch}:=x y\langle h_1,h_2\rangle_{\ch} \, \in E \, .$$

\

Taking the elementary differentiation identity \eqref{deriv-rule-tensor-product} into account, the following definition for the \enquote{NC derivative operator} logically arises :

\begin{definition}\label{def:deriva}
The \emph{derivative operator} $D^X$ on $\ca_X$ is defined as the linear map $D: \ca_X \to \ce([0,T];\ca_X \otimes \ca_X)$ such that $D^X 1=0$ and, if $F:=X(h_1)\cdots X(h_m) \in \ca_X$, 
$$D^X F:=\sum_{i=1}^m \Big[ \big(X(h_1)\cdots X(h_{i-1}) \big) \otimes \big( X(h_{i+1})\cdots X(h_m)\big) \Big] h_i  \, ,$$
using the convention $X(h_1)\cdots X(h_0)=X(h_{m+1})\cdots X(h_m):=1$.
\end{definition}

Note that, along this formalism, one has in particular $D^X X(h)=(1\otimes 1) \, h$, as well as the expected derivation rules
\begin{equation}\label{deriv-p-x-h}
D^X P(X(h))=\partial P (X(h)) \, h \ , \quad D^X(FG)=D^XF \cdot G+F \cdot D^X G \ ,
\end{equation}
where the product $\cdot$ in the latter identity must naturally be understood through (the linear extension of) the formulas
$$(F_1\otimes F_2) \cdot F_3:=F_1\otimes (F_2F_3) \ , \quad F_1 \cdot (F_2\otimes F_3):= (F_1 F_2)\otimes F_3 \ .$$

\

If $F_1,F_2\in \ca_X$, we also define the \emph{partial derivatives} $D^X_i(F_1\otimes F_2) \in \ce([0,T];\ca_X^{\otimes 3})$ ($i=1,2$) as the linear extension of the formulas
$$D^X_1\big( F_1\otimes F_2\big):= D^X F_1\otimes F_2\quad  , \quad D^X_2\big( F_1\otimes F_2\big):= F_1\otimes D^X F_2  \, ,$$
and then set
$$D^X\big( F_1\otimes F_2\big):=D^X_1\big( F_1\otimes F_2\big)+D^X_2\big( F_1\otimes F_2\big)  \  \in \ce([0,T];\ca_X^{\otimes 3}) \, .$$

\

Just as in classical Malliavin calculus, the next challenge is to find a suitable candidate for the \enquote{dual operator} of $D^X$, or otherwise stated the \emph{divergence operator}. Remember that in the commutative setting, the divergence operator $\delta^{\text{com}}$ can be defined on elementary processes as follows (see \cite[Identity (1.44)]{nualart-book}): given a smooth (classical) random variable $F:\Omega \to \R$ and a path $h\in \ce([0,T];\R)$,
\begin{equation}\label{div-comm}
\delta^{\text{com}}(F \, h):=F X(h)-\langle D^{\text{com}}F,h\rangle_\ch \, ,
\end{equation}
where $D^{\text{com}}$ refers to the standard derivative operator (in the commutative setting).

\smallskip

The NC version of \eqref{div-comm} now reads as follows:

\begin{definition}\label{defi:div-op}
For all $\bof \in \ca_X^{\otimes 2}$ and $h\in \ce([0,T];\R)$, we define the \emph{divergence} (or \emph{Skorohod integral}) of $\bof \, h$ with respect to $X$ by the formula
\begin{equation}\label{defi-diver}
\delta^X\big( \bof \, h\big):=\bof \sharp X(h)-\big(\emph{\id} \times \vp \times \emph{\id}\big) \big(\langle D^X \bof,h\rangle_{\ch} \big) \, ,
\end{equation}
and then linearly extend this definition to any $\mathbf{U}\in \ce([0,T];\ca_X^{\otimes 2})$. Note that the operator $\emph{\id}\times \vp \times \emph{\id}$ refers here to the extension of the formula
\begin{equation}\label{conven-prod}
\big(\emph{\id} \times \vp \times \emph{\id}\big) \big(F_1 \otimes F_2\otimes F_3\big):=\vp\big( F_2\big) F_1 F_3 \, .
\end{equation}
\end{definition}

\smallskip

\begin{remark}
The notation in \eqref{conven-prod} follows the simplified convention that we have initiated in \eqref{product-trace}, i.e. we use the symbol $\times$ for the product operation. To be more specific, using the standard notations
$$(\id \otimes \vp \otimes \id ) ( F_1\otimes F_2\otimes F_3):=F_1 \otimes \vp(F_2) \otimes F_3\quad \text{and}\quad m:\ca \otimes \ca \to \ca, \, F\otimes G \mapsto FG \ ,$$
one has here
$$\id \times \vp \times \id=m \circ (\id \otimes \vp \otimes \id) \, .  $$
\end{remark}

\begin{remark}
Consider for instance the case where $T=1$ and $\ch=L^2([0,1])$. Then, for every polynomials $P,Q$ and taking $h:=\mathbf{1}_{[0,1]}$, identity \eqref{defi-diver} reads as
$$\delta^X\big( P(X_1)\otimes Q(X_1)\mathbf{1}_{[0,1]}\big)=P(X_1)X_1Q(X_1)-\big(\id \times \vp \times \id\big)  \big( \partial P(X_1)\otimes Q(X_1)+P(X_1)\otimes \partial Q(X_1)\big) \, ,$$
where we recall that the operator $\partial$ has been defined in \eqref{partial-op}. This shows that, in the non-dynamic setting, Definition \ref{defi:div-op} is in accordance with the result of \cite[Proposition 4.6]{Voi98}.
\end{remark}

\smallskip

As expected, the above definition \eqref{defi-diver} of $\der^X$ is fully justified by a \emph{dual formula}. Remember that in classical probability, the dual formula can be roughly stated (see \cite[Formula (1.42)]{nualart-book}) as
$$\mathbb{E} \big[ Y \delta^{\text{com}}(U)\big]=\mathbb{E}\big[ \langle D^{\text{com}}Y,u\rangle_\ch \big] \, . $$
The NC version of the identity takes a very similar shape:
\begin{proposition}\label{prop:dual-formula}
For all $Y\in \ca_X$ and $\mathbf{U}\in \ce([0,T];\ca_X^{\otimes 2})$, it holds that
\begin{equation}\label{dual-formula}
\vp\big( Y \delta^X\big( \mathbf{U}\big)^\ast \big)=\big( \vp \times \vp\big) \big( \langle D^X Y;\mathbf{U}^{\ast} \rangle_{\ch} \big) \, .
\end{equation}
\end{proposition}

\begin{proof}
See Section \ref{subsec:proof-dual-for}.
\end{proof}


\begin{remark}
Identity \eqref{dual-formula} actually corresponds to a characterization of the divergence operator $\delta^X: \ce([0,T];\ca_X^{\otimes 2}) \to \ca_X$ introduced in Definition \ref{defi:div-op}. In other words, for every fixed $\mathbf{U}\in \ce([0,T];\ca_X^{\otimes 2})$, $\delta^X(\bu)$ is the unique element in $\ca_X$ such that identity \eqref{dual-formula} is satisfied for every $Y\in \ca_X$. Indeed, if an element $Z_{\bu}$ also satisfies this property, we get in particular that $\vp\big( Y (\delta^X(\bu)-Z_{\bu})^\ast\big)=0$ for every $Y\in \ca_X$, and so, choosing $Y:=\delta^X(\bu)-Z_{\bu}$, we can then use the non-degeneracy property of $\vp$ to conclude that $Z_{\bu}=\delta^X(\bu)$.
\end{remark}

\begin{remark}\label{rk:differ-1}
The previous Definition \ref{def:deriva} of the derivative operator immediately coincides with the one provided in \cite{biane-speicher} (see in particular \cite[Proposition 5.2.1]{biane-speicher}). On the other hand, in the latter reference, the definition of the divergence operator $\delta^X$ is only done through the identification of $\ca_X$ with a subspace of the full Fock space (see \cite[Definition 5.1.2]{biane-speicher}). Our Definition \ref{defi:div-op} is thus more intrinsic, and closer to the classical definition of the divergence operator in the commutative case. This remark holds true for the dual formula \eqref{dual-formula} as well. 
\end{remark}

\smallskip

Our next objective is to exhibit some possible isometry property for the (NC) Skorohod integral. In other words, we are here looking for the NC analog of the classical formula (see \cite[Formula (1.45)]{nualart-book})
\begin{equation}\label{iso-comm}
\mathbb{E}\big[\delta^{\text{com}}(U)\delta^{\text{com}}(V)\big]=\mathbb{E}\big[ \langle U,V\rangle_\ch \big]+\mathbb{E}\big[\text{Tr}(D^{\text{com}}U \circ D^{\text{com}}V)\big] \, ,
\end{equation}
where $U$ and $V$ are both smooth (classical) random variables. Let us recall that, in \eqref{iso-comm}, $D^{\text{com}}U$ and $D^{\text{com}}V$ are identified as random Hilbert-Schmidt operators from $\ch$ to $\ch$: the notation $\circ$ then stands for the composition of operators, while the notation $\text{Tr}$ refers to the usual trace of operators.

\smallskip

Here is now the desired NC counterpart of formula \eqref{iso-comm}:

\begin{proposition}\label{prop:ito-isometry}
For all elementary biprocesses $\bu,\bv\in \ce([0,T];\ca_X^{\otimes 2})$, it holds that
\begin{equation}\label{ito-isometry}
\vp\big( \delta^X\big( \bu \big)\delta^X\big( \bv\big)^\ast\big) =\big(\vp\times \vp\big) \big( \langle \bu;\bv^\ast \rangle_\ch \big)+\big( \vp \times \vp \times \vp\big) \big( \mathcal{T}_{\ch}(\bu,\bv \big) \big) \, ,
\end{equation}
where the operator $\mathcal{T}_\ch: \ce([0,T];\ca_X^{\otimes 2}) \times \ce([0,T];\ca_X^{\otimes 2}) \to \ca_X^{\otimes 3}$ is defined as (the bilinear extension of)
$$\mathcal{T}_\ch \big( \mathbf{F}\, h,\mathbf{G}\,  k \big):=\langle D^X_1\mathbf{F},k\rangle_\ch \cdot \langle (D^X_2 \mathbf{G})^\ast ,h \rangle_\ch+\langle D^X_2\mathbf{F},k\rangle_\ch \cdot \langle (D^X_1 \mathbf{G})^\ast ,h \rangle_\ch \, .$$
\end{proposition}

\begin{proof}
See Section \ref{subsec:proof-iso}.
\end{proof}

\smallskip

\begin{remark}
To make the analogy between formulas \eqref{iso-comm} and \eqref{ito-isometry} even more clear, let us notice that in the commutative setting, and when taking $U=F \, h$, $V=G\, k$, with $F,G$ two smooth random variables and $h,k\in \ch$, one has
$$\text{Tr}(D^{\text{com}}U \circ D^{\text{com}}V)=\langle D^{\text{com}}F,k\rangle_\ch \langle D^{\text{com}}G,h\rangle_\ch \, .$$
\end{remark}


\smallskip

\begin{remark}\label{rk:differ-2}
The above Proposition \ref{prop:ito-isometry} corresponds to the extension of the result of \cite[Proposition 5.4.2]{biane-speicher} (which only applies to the free Brownian motion, and in the specific Wigner chaos setting) to a general semicircular process.
\end{remark}

\smallskip

Let us finally recall that in the commutative setting, formula \eqref{iso-comm} immediately provides us with the classical (and very useful) continuous inclusion \enquote{$\mathbb{D}^{1,2}(\ch)\subset \text{Dom} \, \der^{\text{com}}$} (see \cite[Proposition 1.3.1]{nualart-book}). Our final objective in this preliminary section is to exhibit a similar continuous inclusion for the NC setting, starting from formula \eqref{ito-isometry}. For a compact expression of the desired bound, let us set, for all elementary biprocesses $\bu,\bv\in \ce([0,T];\ca_X^{\otimes 2})$,
$$\langle \bu,\bv \rangle_{L^2(\vp\times \vp;\ch)}:=\big(\vp \times \vp\big) \big( \langle \bu;\bv^{\ast }\rangle_{\ch} \big) \, ,$$
and for all elementary triprocesses $\mathbb{U}^{(1)},\mathbb{V}^{(1)},\mathbb{U}^{(2)},\mathbb{V}^{(2)} \in \ce([0,T]^2;\ca_X^{\otimes 3})$, define
\small
\begin{align*}
&\big\langle \begin{pmatrix}
\mathbb{U}^{(1)}\\
\mathbb{V}^{(1)}
\end{pmatrix}, 
\begin{pmatrix}
\mathbb{U}^{(2)}\\
\mathbb{V}^{(2)}
\end{pmatrix}\big\rangle_{L^2(\vp\times \vp\times \vp;\ch\otimes \ch)^2}\\
&:=\big( \vp\times \vp\times \vp\big)\big(\big\langle \mathbb{U}^{(1)},\mathbb{U}^{(2)}\big\rangle_{\ch\otimes \ch}\big)
+\big( \vp\times \vp\times \vp\big)\big(\big\langle \mathbb{V}^{(1)},\mathbb{V}^{(2)}\big\rangle_{\ch\otimes \ch}\big)\, ,
\end{align*}
\normalsize
where the inner product $\langle.,.\rangle_{\ch\otimes \ch}$ refers to (the multilinear extension of)
\begin{align*}
&\big\langle \big((s,t)\mapsto\mathbb{F}\, h_s\, k_t\big), \big( (s,t)\mapsto\mathbb{G} \, \ell_s\, m_t\big)\big\rangle_{\ch\otimes \ch}:=\big( \mathbb{F}\cdot\mathbb{G}^\ast \big) \langle h,\ell\rangle_{\ch} \langle k,m\rangle_\ch \, .
\end{align*}
Finally, for every elementary biprocess $\bu\in \ce([0,T];\ca_X^{\otimes 2})$, let us set
$$\nabla^X \bu:=\begin{pmatrix} D^X_1 \bu\\ D^X_2\bu \end{pmatrix} \in \ce([0,T]^2;\ca_X^{\otimes 3})^2 \, .$$
We are now ready to state the expected estimate (see \cite[Estimate (1.47)]{nualart-book} for the commutative counterpart of this property):
\begin{corollary}\label{coro:bound-ito}
For every elementary biprocess $\bu\in \ce([0,T];\ca_X^{\otimes 2})$, it holds that
\begin{equation}\label{bound-ito}
\big\| \delta^X(\bu)\big\|_{L^2(\vp)} \leq \big\| \bu\big\|_{\mathbb{D}^{1,2}(\ch)} \, ,
\end{equation}
where $\big\| \bu\big\|_{\mathbb{D}^{1,2}(\ch)}$ is defined as
\begin{equation}\label{defi:d-1-2}
\big\| \bu\big\|_{\mathbb{D}^{1,2}(\ch)}^2:=\big\| \bu \big\|_{L^2(\vp\times \vp;\ch)}^2+\big\| \nabla^X \bu \big\|_{L^2(\vp\times \vp\times \vp;\ch\otimes \ch)^2}^2 \, .
\end{equation}
\end{corollary}

\begin{proof}
See Section \ref{subsec:proof-coro-boun}.
\end{proof}

Before we turn to the details of the proofs of the above properties, and as a conclusion to this presentation, let us adapt the classical definition of Skorohod integrability to the NC setting, along the following simple formulation:
\begin{definition}\label{defi:skoro-int}
For all $0\leq s\leq t\leq T$, we will say that a biprocess $\mathbf{U}:[0,T]\to \ca_X \otimes \ca_X$ is Skorohod-integrable with respect to $X$ on $[s,t]$ if, for any subdivision $\Delta_{st}:=\{s=r_0<\ldots <r_n=t\}$ with mesh $|\Delta_{st}|$ tending to $0$, and setting
\begin{equation}\label{bu-n}
\mathbf{U}^{\Delta_{st}}:=\sum_{i=0}^{n-1} \mathbf{U}_{r_i} \1_{[r_i,r_{i+1})} \, ,
\end{equation}
the sequence $\delta^X\big( \mathbf{U}^{\Delta_{st}} \big)$ converges in $L^2(\vp)$ as $|\Delta_{st}|\to 0$. In this case, we will denote the limit by $\delta^X_{s,t}\big( \mathbf{U}\big)$.
\end{definition}

\smallskip

As an immediate consequence of Corollary \ref{coro:bound-ito}, we derive the following practical integrability criterion (the analog of the classical continuous inclusion \enquote{$\mathbb{D}^{1,2}(\ch)\subset \text{Dom} \, \der^{\text{com}}$}):

\begin{proposition}\label{prop:crit-sko-int}
Consider a biprocess $\mathbf{U}:[0,T]\to \ca_X \otimes \ca_X$, and $0\leq s\leq t\leq T$. If, for any subdivision $\Delta_{st}:=\{s=r_0<\ldots <r_n=t\}$ with mesh $|\Delta_{st}|$ tending to $0$, the approximation $\bu^{\Delta_{st}}$ (defined by (\ref{bu-n})) is a Cauchy sequence with respect to $\|.\|_{\mathbb{D}^{1,2}(\ch)}$ (defined by (\ref{defi:d-1-2})), then $\bu$ is Skorohod-integrable with respect to $X$ on $[s,t]$.
\end{proposition}

\



\subsection{Proof of Proposition \ref{prop:dual-formula}}\label{subsec:proof-dual-for}

\

\smallskip

Consider $Y=X(h_1)\cdots X(h_m)$ and $\bu=\bof h$, with $\bof=F_1\otimes F_2$, $F_1=X(f_1)\cdots X(f_n)$, $F_2=X(f_{n+1})\cdots X(f_{n+p})$. Then one has
$$\delta^X(\bu)=F_1X(h)F_2-\big( \id \times \vp\times \id\big) \big( \langle D^X\bof ,h\rangle_{\ch}\big) \, ,$$
with the explicit expansion
\begin{align}
&\Big(\big(\id \times \vp \times \id\big) \big(\langle D^X \bof,h\rangle_{\ch} \big)\Big)^\ast\nonumber\\
&=\bigg(\sum_{j=1}^n  \vp\big( X(f_{j+1})\cdots X(f_n) \big)\Big[ \big( X(f_1)\cdots X(f_{j-1})\big)  F_2 \Big] \vp\big( X(f_j)X(h)\big)\nonumber\\
&\hspace{0.5cm}+\sum_{k=1}^p \vp\big( X(f_{n+1})\cdots X(f_{n+k-1}) \big)\Big[ F_1  \big( X(f_{n+k+1}) \cdots X(f_{n+p}) \big) \Big] \vp\big( X(f_{n+k})X(h)\big)\bigg)^\ast\nonumber\\
&=\sum_{j=1}^n  \vp\big( X(f_{n})\cdots X(f_{j+1}) \big)\Big[F_2^\ast  \big( X(f_{j-1})\cdots X(f_{1})\big)  \Big] \vp\big( X(f_j)X(h)\big)\nonumber\\
&\hspace{0.5cm}+\sum_{k=1}^p \vp\big( X(f_{n+k-1})\cdots X(f_{n+1}) \big)\Big[ \big( X(f_{n+p}) \cdots X(f_{n+k+1}) \big)F_1^\ast \Big] \vp\big( X(f_{n+k})X(h)\big) \, .\label{i-tim-phi-tim-i}
\end{align}
Using the NC Wick formula (and more specifically the result of Proposition \ref{prop:wick}), we can write
\begin{align*}
&\vp\big(Y F_2^\ast X(h)F_1^\ast \big)=\vp\Big(\big(X(h_1)\cdots X(h_m)\big) \big(X(f_{n+p})\cdots X(f_{n+1})\big)X(h) \big(X(f_{n})\cdots X(f_{1})\big)   \Big)\\
&=\sum_{i=1}^m \vp\big( X(h_i)X(h)\big) \vp\big( \big(X(h_1)\cdots X(h_{i-1})\big)F_1^\ast\big)\vp\big( \big(X(h_{i+1})\cdots X(h_m)\big)F_2^\ast \big)\\
&\hspace{0.5cm}+\sum_{k=1}^p \vp\big( X(f_{n+k})X(h)\big)\vp\big(Y\big( X(f_{n+p})\cdots X(f_{n+k+1})\big)F_1^\ast\big) \vp\big( X(f_{n+k-1})\cdots X(f_{n+1}) \big)\\
&\hspace{0.5cm}+\sum_{j=1}^n \vp\big(X(f_j)X(h)\big) \vp\big( YF_2^\ast \big(X(f_{j-1})\cdots X(f_{1})\big) \big) \vp\big( X(f_{n}) \cdots X(f_{j+1})\big)
\end{align*}
and so, by comparing this expansion with \eqref{i-tim-phi-tim-i}, we deduce
\begin{align*}
&\vp\Big( Y \delta^X\big( \mathbf{U}\big)^\ast \Big)=\vp\big(Y F_2^\ast X(h)F_1^\ast \big)-\vp\Big( Y \Big(\big(\id \times \vp \times \id\big) \big(\langle D^X \bof,h\rangle_{\ch}\big) \Big)^\ast \Big)\\
&=\sum_{i=1}^m \vp\big( X(h_i)X(h)\big) \vp\big( \big(X(h_1)\cdots X(h_{i-1})\big)F_1^\ast\big)\vp\big( \big(X(h_{i+1})\cdots X(h_m)\big)F_2^\ast \big)\\
&=\big(\vp \times \vp\big) \bigg( \sum_{i=1}^m \big( \big(X(h_1)\cdots X(h_{i-1})\big)F_1^\ast\big) \otimes \big( \big(X(h_{i+1})\cdots X(h_m)\big)F_2^\ast \big) \langle h_i,h\rangle_{\ch} \bigg)\\
&=\big(\vp \times \vp\big) \bigg( \bigg\langle \sum_{i=1}^m \Big[\big( X(h_1)\cdots X(h_{i-1})\big) \otimes \big(X(h_{i+1})\cdots X(h_m) \big) \Big] h_i;\big(F_1^\ast \otimes F_2^\ast\big) h\bigg\rangle_{\ch} \bigg)\\
&=\big( \vp \times \vp\big) \big( \langle D^X Y;\mathbf{U}^{\ast} \rangle_{\ch} \big) \, ,
\end{align*}
which corresponds to the desired identity (\ref{dual-formula}).

\begin{flushright}
$\square$
\end{flushright}

\

\subsection{Proof of Proposition \ref{prop:ito-isometry}}\label{subsec:proof-iso}

\

\smallskip

It suffices to prove formula \eqref{ito-isometry} for elements $\bu,\bv$ of the form
$$\bu:=(F_1 \otimes F_2) \, h \quad , \quad \bv:=(G_1 \otimes G_2)\, k \, ,$$
with
$$F_1:=X(f_1) \cdots X(f_m) \quad , \quad F_2:=X(f_{m+1})\cdots X(f_{m+M}) \, ,$$
and
$$G_1:=X(g_1)\cdots X(g_n) \quad , \quad G_2:=X(g_{n+1})\cdots X(g_{n+N}) \, ,$$
for some elementary functions $h,f_1,\ldots f_{m+M}$ and $k,g_1,\ldots ,g_{n+N}$.

\smallskip

We can first apply the dual formula (\ref{dual-formula}) to write
\begin{equation}\label{dual-form-appli}
\vp\big( \delta^X(\bu) \delta^X(\bv)^\ast \big)=\big( \vp \times \vp \big) \big( \langle D^X \delta^X(\bu); \bv^\ast \rangle_{\ch} \big) \ .
\end{equation}
Let us then recall that
$$\der^X(\bu)=F_1 X(h)F_2-\big( \id \times \vp \times \id \big)\big( \langle D^X(F_1\otimes F_2),h \rangle_{\ch} \big) \ .$$
Using the derivation rules in \eqref{deriv-p-x-h}, we get that
$$D^X(F_1 X(h)F_2 )=D^X F_1 \cdot (X(h)F_2)+(F_1 \otimes F_2) h+(F_1 X(h))\cdot D^X F_2 \, ,$$
which yields the decomposition
$$D^X \delta^X(\bu)=\bu+\mathcal{R}(\bu) \ ,$$
with $\mathcal{R}(\bu):=\mathcal{R}_1(\bu)-\mathcal{R}_2(\bu)$, 
$$\mathcal{R}_1(\bu):=D^XF_1 \cdot (X(h)F_2)+(F_1 X(h))\cdot D^XF_2 \, $$
and
$$\mathcal{R}_2(\bu):=D^X\Big[ \big( \id \times \vp\times \id \big) \big( \langle D^X(F_1 \otimes F_2),h \rangle_\ch \big)\Big] \ .$$
Going back to (\ref{dual-form-appli}), we see that we are now left with the proof of the identity
\begin{equation}\label{desired-formula}
\big( \vp \times \vp \big) \big(\mathcal{R}(\bu) ; \bv^\ast \rangle_{\ch} \big)=\big( \vp \times \vp \times \vp\big) \big( \mathcal{T}_{\ch}(\bu,\bv \big) \big)\, .
\end{equation}




\smallskip

On the one hand, using only the very definition of $D^X$, one gets the expansion
\begin{align*}
&\big( \vp\times \vp\big) \big(\langle \mathcal{R}_1(\bu);\bv^\ast \rangle_{\ch}\big)\\
&=\sum_{i=1}^m \vp(X(f_i)X(k))\vp\big( X(f_1)\cdots X(f_{i-1})G_1^\ast\big) \vp\big( X(f_{i+1})\cdots X(f_m) X(h)F_2 G_2^\ast \big) \\
&+\sum_{j=1}^M \vp(X(f_{m+j})X(k)) \vp\big( F_1 X(h) X(f_{m+1})\cdots X(f_{m+j-1})G_1^\ast \big) \vp\big( X(f_{m+j+1})\cdots X(f_{m+M})G_2^\ast \big)   \, ,
\end{align*}
and applying the NC Wick formula (as stated in Proposition \ref{prop:wick}) to the expectation terms 
$$\vp\big( X(f_{i+1})\cdots X(f_m) X(h)F_2 G_2^\ast \big) \quad \text{and} \quad \vp\big( F_1 X(h) X(f_{m+1})\cdots X(f_{m+j-1})G_1^\ast \big) \, ,$$
we derive the decomposition
\begin{equation}\label{decompo-r-1}
\big( \vp\times \vp\big) \big(\langle \mathcal{R}_1(\bu);\bv^\ast \rangle_{\ch}\big)=I+II \, ,
\end{equation}
with
\begin{align*}
I:=&\sum_{i=1}^m \sum_{i'=i+1}^m  \vp(X(f_i)X(k))\vp(X(f_ {i'})X(h))\vp\big( X(f_1)\cdots X(f_{i-1})G_1^\ast\big) \\
&\hspace{4cm}\vp\big( X(f_{i+1})\cdots X(f_{i'-1})F_2 G_2^\ast \big)\vp\big( X(f_{i'+1})\cdots X(f_m)\big)\\
&+\sum_{i=1}^m\sum_{j=1}^M \vp(X(f_i)X(k))\vp(X(h)X(f_{m+j}))\vp\big( X(f_1)\cdots X(f_{i-1})G_1^\ast\big) \\
&\hspace{1cm}\vp\big( X(f_{i+1})\cdots X(f_m)X(f_{m+j+1})\cdots X(f_{m+M})G_2^\ast \big)\vp\big( X(f_{m+1})\cdots X(f_{m+j-1})\big)\\
&+\sum_{i=1}^m\sum_{k=0}^{N-1}  \vp(X(f_i)X(k)) \vp(X(h)X(g_{n+N-k}))\vp\big( X(f_1)\cdots X(f_{i-1})G_1^\ast\big) \\
&\hspace{0.5cm}\vp\big( X(f_{i+1}) \cdots X(f_m)X(g_{n+N-k-1})\cdots X(g_{n+1}) \big) \vp\big( F_2X(g_{n+N})\cdots X(g_{n+N-k+1})\big) 
\end{align*}
and
\begin{align*}
II:=&\sum_{j=1}^M \sum_{i=1}^m  \vp(X(f_{m+j})X(k)) \vp(X(f_i)X(h))  \vp\big( X(f_{m+j+1})\cdots X(f_{m+M})G_2^\ast \big)\\
&\hspace{1cm}\vp\big( X(f_1)\cdots X(f_{i-1})X(f_{m+1}) \cdots  X(f_{m+j-1})G_1^\ast \big)\vp\big(X(f_{i+1})\cdots X(f_m)\big)\\
&+\sum_{j=1}^M \sum_{j'=1}^{j-1} \vp(X(f_{m+j})X(k)) \vp(X(h)X(f_{m+j'}))  \vp\big( X(f_{m+j+1})\cdots X(f_{m+M})G_2^\ast \big)\\
&\hspace{1cm}\vp\big( F_1 X(f_{m+j'+1}) \cdots X(f_{m+j-1})G_1^\ast \big) \vp\big( X(f_{m+1})\cdots X(f_{m+j'-1})\big)\\
&+\sum_{j=1}^M \sum_{\ell=0}^{n-1} \vp(X(f_{m+j})X(k)) \vp(X(h)X(g_{n-\ell})) \vp\big( X(f_{m+j+1})\cdots X(f_{m+M})G_2^\ast \big)\\
&\hspace{1cm}\vp\big( F_1 X(g_{n-\ell-1}) \cdots X(g_1) \big)\vp\big( X(f_{m+1})\cdots X(f_{m+j-1})X(g_n)\cdots X(g_{n-\ell+1})\big) \ .
\end{align*}

\

\

\noindent

On the other hand, using again the very definition of $D^X$, we can readily expand the second quantity $\big( \vp \times \vp \big) \big(\langle\mathcal{R}_2(\bu) ; \bv^\ast \rangle_{\ch} \big)$ as
\begin{align*}
&\big( \vp \times \vp \big) \big(\langle\mathcal{R}_2(\bu) ; \bv^\ast \rangle_{\ch} \big)\\
&= \sum_{i=1}^m \sum_{i'=1}^{i-1}\vp(X(f_i)X(h))\vp(X(f_{i'})X(k)) \vp\big( X(f_{i+1})\cdots X(f_m)\big)\\
&\hspace{4cm}  \vp\big( X(f_1)\cdots X(f_{i'-1})G_1^\ast\big) \vp\big( X(f_{i'+1})\cdots X(f_{i-1}) F_2G_2^\ast\big) \\
&\hspace{0.5cm}+\sum_{i=1}^m \sum_{j=1}^{M}\vp(X(f_i)X(h)) \vp(X(f_{m+j})X(k))\vp\big( X(f_{i+1})\cdots X(f_m)\big)\\
&\hspace{1cm}\vp\big( X(f_1)\cdots X(f_{i-1})X(f_{m+1})\cdots X(f_{m+j-1})G_1^\ast\big) \vp\big( X(f_{m+j+1})\cdots X(f_{m+M})G_2^\ast\big) \\
&+ \sum_{j=1}^M\sum_{i=1}^m  \vp(X(f_{m+j})X(h))\vp(X(f_i)X(k)) \vp\big( X(f_{m+1})\cdots X(f_{m+j-1})\big)\\
&\hspace{2cm} \vp\big(X(f_1)\cdots X(f_{i-1})G_1^\ast\big) \vp\big( X(f_{i+1})\cdots X(f_m) X(f_{m+j+1})\cdots X(f_{m+M})G_2^\ast\big)\\
&+ \sum_{j=1}^M\sum_{j'=j+1}^M  \vp(X(f_{m+j})X(h))\vp(X(f_{m+j'})X(k)) \vp\big( X(f_{m+1})\cdots X(f_{m+j-1})\big)\\
&\hspace{2cm}\vp\big( F_1X(f_{m+j+1})\cdots X(f_{m+j'-1})G_1^\ast\big) \vp\big( X(f_{m+j'+1})\cdots X(f_{m+M})G_2^\ast\big)  \, .
\end{align*}
Comparing the latter expansion with (\ref{decompo-r-1}), we easily get that
$$\big( \vp\times \vp\big) \big(\langle \mathcal{R}(\bu);\bv^\ast \rangle_{\ch}\big)=\big( \vp\times \vp\big) \big(\langle \mathcal{R}_1(\bu);\bv^\ast \rangle_{\ch}\big)-\big( \vp\times \vp\big) \big(\langle \mathcal{R}_2(\bu);\bv^\ast \rangle_{\ch}\big)=\mathcal{Q}_1+\mathcal{Q}_2 \, ,$$
with
\begin{align*}
\mathcal{Q}_1&:=\sum_{i=1}^m \sum_{k=0}^{N-1} \vp\big( X(f_1)\cdots X(f_{i-1})G_1^\ast\big) \vp\big( X(f_{i+1}) \cdots X(f_m)X(g_{n+N-k-1})\cdots X(g_{n+1}) \big)\\
&\hspace{3cm} \vp\big( F_2X(g_{n+N})\cdots X(g_{n+N-k+1})\big) \, \langle f_i,k\rangle_{\ch} \langle h,g_{n+N-k}\rangle_{\ch}\\
&=\sum_{k=0}^{N-1} \big( \vp\times \vp\big) \big( \langle D^X F_1,k \rangle_{\ch} \cdot (G_1^\ast\otimes ( X(g_{n+N-k-1})\cdots X(g_{n+1}))) \big)\\
&\hspace{5cm} \vp\big( F_2X(g_{n+N})\cdots X(g_{n+N-k+1})\big) \,  \langle h,g_{n+N-k}\rangle_{\ch}\\
&=\big( \vp\times \vp\times \vp\big) \big( \langle D^X F_1 \otimes F_2,k\rangle_{\ch} \cdot \langle G_1^\ast\otimes (D^X G_2)^\ast, h\rangle_{\ch}\big) \, ,
\end{align*}
and similarly
\begin{align*}
\mathcal{Q}_2&:=\sum_{j=1}^M \sum_{\ell=0}^{n-1} \vp\big(F_1X(g_{n-\ell-1})\cdots X(g_1) \big) \vp\big(X(f_{m+1})\cdots X(f_{m+j-1})X(g_n)\cdots X(g_{n-\ell+1}) \big)\\
&\hspace{3cm} \vp\big(X(f_{m+j+1})\cdots X(f_{m+M})G_2^\ast \big) \, \langle f_{m+j},k\rangle_{\ch} \langle h,g_{n-\ell}\rangle_{\ch}\\
&=\big( \vp\times \vp\times \vp\big) \big( \langle F_1 \otimes D^X F_2,k \rangle_{\ch} \cdot \langle (D^X G_1)^\ast\otimes G_2^\ast, h\rangle_{\ch}\big) \, .
\end{align*}
This corresponds to the desired formula (\ref{desired-formula}).

\smallskip

\begin{flushright}
$\square$
\end{flushright}

\

\subsection{Proof of Corollary \ref{coro:bound-ito}}\label{subsec:proof-coro-boun}

\

\smallskip

First, observe that
\begin{equation}\label{t-tilde-t}
\mathcal{T}_\ch\big(\bu,\bu\big)=\widetilde{\mathcal{T}}_\ch\big( \nabla^X \bu \big) \,
\end{equation}
where $\widetilde{\mathcal{T}}_\ch$ is defined as (the multilinear extension of)
$$\widetilde{\mathcal{T}}_\ch \begin{pmatrix} (s,t)\mapsto \mathbb{F} \, h_s \, k_t \\
(s,t)\mapsto \mathbb{G}\, \ell_s \, m_t \end{pmatrix}:=\big[ \mathbb{F}\cdot \mathbb{G}^\ast+\mathbb{G}\cdot \mathbb{F}^\ast \big] \langle h,m\rangle_\ch \langle \ell,k\rangle_\ch \, .$$ 
Now, for all $\mathbb{U}_{s,t}:=\sum_i \mathbb{F}^i \, h^i_s\, k^i_t$ and $\mathbb{V}:=\sum_j \mathbb{G}^j\, \ell^j_s\, m^j_t$,
\begin{align*}
&\big( \vp\times \vp \times \vp\big)\left( \widetilde{\mathcal{T}}_\ch \begin{pmatrix} \mathbb{U} \\
\mathbb{V} \end{pmatrix}\right)=\sum_{i,j} \big( \vp\times \vp \times \vp\big)\big( \mathbb{F}^i\cdot \big( \mathbb{G}^j\big)^\ast+\mathbb{G}^j\cdot \big(\mathbb{F}^i\big)^\ast \big) \langle h^i,m^j\rangle_\ch \langle  \ell^j,k^i\rangle_\ch\\
&=\sum_{i,j} \big\langle \mathbb{F}^i,\mathbb{G}^j\big\rangle_{L^2(\vp\times \vp\times \vp)} \langle h^i,m^j\rangle_\ch \langle k^i,\ell^j\rangle_\ch+\sum_{i,j} \big\langle \mathbb{G}^j,\mathbb{F}^i\big\rangle_{L^2(\vp\times \vp\times \vp)} \langle m^j,h^i\rangle_\ch \langle \ell^j,k^i\rangle_\ch\\
&=\sum_{i,j} \big\langle \mathbb{F}^i \otimes h^i \otimes k^i,\mathbb{G}^j\otimes m^j \otimes \ell^j\big\rangle_{L^2(\vp\times \vp\times \vp)\otimes \ch \otimes \ch}\\
&\hspace{2cm}+\sum_{i,j} \big\langle \mathbb{G}^j\otimes m^j \otimes \ell^j, \mathbb{F}^i \otimes h^i \otimes k^i\big\rangle_{L^2(\vp\times \vp\times \vp)\otimes \ch \otimes \ch}\\
&= \big\langle \mathbb{U},\mathbb{V}\big\rangle_{L^2(\vp\times \vp\times \vp)\otimes \ch \otimes \ch} +\big\langle \mathbb{V},\mathbb{U}\big\rangle_{L^2(\vp\times \vp\times \vp)\otimes \ch \otimes \ch}\, ,
\end{align*}
and so
\begin{align*}
&\bigg|\big( \vp\times \vp \times \vp\big)\left( \widetilde{\mathcal{T}}_\ch \begin{pmatrix} \mathbb{U} \\
\mathbb{V} \end{pmatrix}\right)\bigg|\\
&\hspace{1cm}\leq 2 \big\| \mathbb{U}\big\|_{L^2(\vp\times \vp\times \vp)\otimes \ch \otimes \ch}\big\| \mathbb{V}\big\|_{L^2(\vp\times \vp\times \vp)\otimes \ch \otimes \ch}\\
&\hspace{1cm}\leq \big\| \mathbb{U}\big\|_{L^2(\vp\times \vp\times \vp)\otimes \ch \otimes \ch}^2+\big\| \mathbb{V}\big\|_{L^2(\vp\times \vp\times \vp)\otimes \ch \otimes \ch}^2=\big\| \begin{pmatrix}
\mathbb{U}\\
\mathbb{V}
\end{pmatrix}\big\|_{L^2(\vp\times \vp\times \vp;\ch\otimes \ch)^2} \, .
\end{align*}
Combining the latter estimate with formulas \eqref{ito-isometry} and \eqref{t-tilde-t} yields the desired bound \eqref{bound-ito}.
\begin{flushright}
$\square$
\end{flushright}

\

\section{Integration with respect to the non-commutative fractional Brownian motion}\label{sec:int-nc-fbm}

From now on and for the rest of the paper, we will specialize our analysis to the case of the \emph{NC fractional Brownian motion} (NC-fBm in the sequel). 
\smallskip

This model was already at the core of our considerations in \cite{deya-schott-3} (see also \cite{deya-hfbm}), and it provides us with a natural extension of the celebrated free Brownian motion. For the sake of completeness, let us briefly recall that the NC-fBm is a specific family of (centered) semicircular processes. As such, these processes are fully characterized by their covariance function (just as centered Gaussian processes in the classical setting), and we can therefore fully describe the model as follows:

\begin{definition}\label{defi:NC-fbm}
Fix a NC probability space $(\ca,\vp)$, as well as a finite time horizon $T>0$. For every $H\in (0,1)$, we call a NC fractional Brownian motion (NC-fBm) of Hurst index $H$ any centered semicircular process $\{X_t\}_{t\in [0,T]}$ in $(\ca,\vp)$ whose covariance function is given by the classical fractional formula
\begin{equation}\label{cova-NC-fBm}
\vp\big( X_{s}X_{t}\big)=R_H(s,t):=\frac12 \big\{s^{2H}+t^{2H}-|t-s|^{2H}\big\} \ , \ s,t\in [0,T] \, .
\end{equation}
\end{definition}

Following this definition, it is easy to see that the NC-fBm of Hurst index $\frac12$ is nothing but the celebrated free Brownian motion, for which NC stochastic calculus was originally developed (in \cite{biane-speicher}). Let us also recall that as soon as $H\neq \frac12$, the fundamental free independence property of the disjoint increments is lost, leaving us with major technical difficulties regarding integration with respect to the NC-fBm. In particular, we can no longer rely on the It{\^o}-type arguments used in \cite{biane-speicher}.

\smallskip

In this context, and with the considerations of the previous section in mind, our first objective will be to show that the Skorohod approach (i.e. Definition \ref{defi:skoro-int}) can still be applied in the presence of the NC-fBm, at least for any Hurst index $H>\frac14$ and for a class of simple (but non-trivial) integrands, providing us with a possible natural interpretation of the integral in this case.

\

In fact, in the rest of the paper, and for obvious technical reasons, we will restrict our attention to polynomial integration, that is we consider integrands of the form $t\mapsto P(X_t)\otimes Q(X_t)$, for two polynomials $P,Q$. Note that this restriction already prevailed in \cite{deya-schott-3}.

\subsection{Skorohod integration with respect to the NC-fBm}\label{subsec:skoro-fbm}

\

\smallskip

The main statement in this setting should not come as a surprise:

\begin{proposition}\label{prop:skoro-integr}
Assume that $\{X_t, \, t\in [0,T]\}$ is a NC-fBm of Hurst index $H\in (\frac14,1)$, in a given NC probability space $(\ca,\vp)$. Then, for all $0\leq s\leq t\leq T$ and all polynomials $P,Q$, the biprocess $u \mapsto P(X_u)\otimes Q(X_u)$ is Skorohod-integrable with respect to $X$ on $[s,t]$ (in the sense of Definition \ref{defi:skoro-int}).
\end{proposition}

As the reader might expect it, our strategy to prove Proposition \ref{prop:skoro-integr} will rely on the use of the practical criterion exhibited in Proposition \ref{prop:crit-sko-int}. A first step here consists in the estimation of the abstract norms $\big\| . \big\|_{L^2(\vp\times \vp;\ch)}$ and $\big\|. \big\|_{L^2(\vp\times \vp\times \vp;\ch\otimes \ch)}$ (involved in the definition \eqref{defi:d-1-2} of $\big\|.\big\|_{\mathbb{D}^{1,2}(\ch)}$) in terms of more explicit quantities, depending on $H$ only. To this end, we shall in fact lean on similar bounds as in the commutative case, as detailed in Lemma \ref{lem:identif-space} below.

\

\textit{For the sake of conciseness, we will assume for the whole proof of Proposition \ref{prop:skoro-integr} (and therefore for the rest of Section \ref{subsec:skoro-fbm}) that $H\leq  \frac12$. However, it is not hard to check that the proof in the (easier) case $H\in (\frac12,1)$ can be derived from minor adaptations of Lemma \ref{lem:identif-space} below (according to the estimates in \cite[Chapter 5]{nualart-book} for the fractional kernel $K_H$).}

\smallskip

\begin{lemma}\label{lem:identif-space}
Fix $H\in (0,\frac12]$, and let $\ch$ be the space associated with the NC-fBm of Hurst index $H$ (through \eqref{defi-ch}). Then there exists a constant $c_H>0$ such that, for every elementary biprocess $\bu\in \ce([0,T];\ca^{\otimes 2})$, one has
\begin{equation}\label{identif-space}
\big\| \bu\big\|_{\mathbb{D}^{1,2}(\ch)} \leq c_H \big\{\big\| \bu \big\|_{1;T,H,\vp} +\big\| D^X_1 \bu \big\|_{2;T,H,\vp}+\big\| D^X_2 \bu \big\|_{2;T,H,\vp} \big\} \, , 
\end{equation}
where 
\begin{equation*}
\big\|\bu\big\|_{1;T,H,\vp}^2:=\int_0^T du \, |T-u|^{2H-1} \big\| \bu_u\big\|_{L^2(\vp\times \vp)}^2+\int_0^T du \bigg( \int_u^T dv \, |v-u|^{H-\frac32} \big\| \bu_v-\bu_u\big\|_{L^2(\vp \times \vp)} \bigg)^2 
\end{equation*}
and, for every $\mathbb{U}\in \ce([0,T]^2; \ca^{\otimes 3})$,
\begin{align*}
&\big\| \U \big\|_{2;T,H,\vp}^{2}:=\int_{[0,T]^2} du_1 du_2 \, |T-u_1|^{2H-1} |T-u_2|^{2H-1} \big\| \U_{u_1,u_2}\big\|_{L^2(\vp\times \vp\times \vp)}^2\\
&+ \int_{[0,T]^2} du_1 du_2 \, |T-u_1|^{2H-1}  \bigg(\int_{u_2}^T dv \, |v-u_2|^{H-\frac32} \big\| \U_{u_1,v}-\U_{u_1,u_2}\big\|_{L^2(\vp\times \vp\times \vp)}^2 \bigg)^2\\
&+ \int_{[0,T]^2} du_1 du_2 \, |T-u_2|^{2H-1}  \bigg(\int_{u_1}^T dv \, |v-u_1|^{H-\frac32} \big\| \U_{v,u_2}-\U_{u_1,u_2}\big\|_{L^2(\vp\times \vp\times \vp)}^2 \bigg)^2\\
&+ \int_{[0,T]^2} du_1 du_2\\
&\hspace{0.5cm}   \bigg(\int_{u_1}^T dv_1\int_{u_2}^T dv_2 \, |v_1-u_1|^{H-\frac32}|v_2-u_2|^{H-\frac32} \big\| \U_{v_1,v_2}-\U_{v_1,u_2}-\U_{u_1,v_2}+\U_{u_1,u_2}\big\|_{L^2(\vp\times \vp\times \vp)}^2 \bigg)^2 \, .
\end{align*}
\end{lemma}

\begin{proof}
Just as in the classical commutative case, the bound relies on the consideration of the kernel $K_H$ defined in \cite[Proposition 5.1.3]{nualart-book}, and which satisfies, for all elementary functions $h,k\in \ce([0,T];\R)$, 
\begin{equation}\label{iso-k-h}
\langle h,k\rangle_{\ch}=\langle K_{H,T}^\ast h,K_{H,T}^\ast k\rangle_{L^2([0,T])}
\end{equation}
with
$$(K_{H,T}^\ast h)(u):=K_H(T,u) h(u)+\int_u^T dv \, \frac{\partial K_{H}}{\partial v} (v,u) \{h(v)-h(u)\} \, ,$$
as well as
\begin{equation}\label{estim-k-h}
|K_H(v,u)| \leq c_H \, |v-u|^{H-\frac12} \quad , \quad \Big| \frac{\partial K_H}{\partial v}(v,u)\Big| \leq c_H \, |v-u|^{H-\frac32} \, ,
\end{equation}
for all $0\leq u<v\leq T$ and for some constant $c_H>0$. 

\smallskip

In our setting, and given $\bu \in \ce([0,T];\ca^{\otimes 2})$, we can use the isometry \eqref{iso-k-h} to write $\big\| \bu \big\|_{L^2(\vp\times \vp;\ch)}^2$ as
\begin{align*}
&\big\| \bu \big\|_{L^2(\vp\times \vp;\ch)}^2 =\int_0^T du \, K_H(T,u)^2 \big( \vp\times \vp\big) \big( \bu_u \cdot \bu^\ast_u \big)\\
&\hspace{0.5cm}+\int_0^T du  \, K_H(T,u) \int_u^T dv\frac{\partial K_H}{\partial v}(v,u) \, \big( \vp\times \vp\big) \Big( \bu_u \cdot \big( \bu_v-\bu_u\big)^\ast  \Big)\\
&\hspace{0.5cm}+\int_0^T du  \, K_H(T,u) \int_u^T dv\frac{\partial K_H}{\partial v}(v,u) \, \big( \vp\times \vp\big) \Big( \big( \bu_v-\bu_u\big) \cdot \bu_u^\ast  \Big)\\
&\hspace{0.5cm}+\int_0^T du \int_u^T dv \, \frac{\partial K_H}{\partial v}(v,u) \int_u^T d\tilde{v} \, \frac{\partial K_H}{\partial v}(\tilde{v},u) \, \big( \vp\times \vp\big) \Big( \big( \bu_v-\bu_u\big) \cdot \big( \bu_{\tilde{v}}-\bu_u \big)^\ast \Big) \ .
\end{align*}

The desired bound $\big\| \bu \big\|_{L^2(\vp\times \vp;\ch)} \leq c_H \big\| \bu \big\|_{1;H,T,\vp}$ is then a straightforward consequence of the Cauchy-Schwarz inequality (for $\vp\times \vp$), combined with the two estimates in \eqref{estim-k-h}.

\smallskip

Similar arguments can then be used in order to show that
$$\big\| \U \big\|_{L^2(\vp\times \vp \times \vp;\ch \otimes \ch)} \leq c_H \big\| \U \big\|_{2;H,T,\vp}\, , $$
for every $\mathbb{U}\in \ce([0,T]^2; \ca^{\otimes 3})$.
\end{proof}

\smallskip

As a second step toward Proposition \ref{prop:skoro-integr}, and with the result of Lemma \ref{lem:identif-space} in mind, consider a general Banach space $(E,\|.\|_E)$ and for any path $x:[0,T] \to E$, let us define the quantity $\|x\|_{1;H,T,E}$ by replacing $\bu$ with $x$ and $\|.\|_{L^2(\vp\times \vp)}$ with $\|.\|_E$ in the definition of $\|.\|_{1;H,T,\vp}$. In the same vein, and for any path $\mathbf{x}:[0,T]^2 \to E$, let us define the quantity $\|\mathbf{x}\|_{2;H,T,E}$ by replacing $\U$ with $\mathbf{x}$ and $\|.\|_{L^2(\vp\times \vp \times \vp)}$ with $\|.\|_E$ in the definition of $\|.\|_{2;H,T,\vp}$.

\smallskip

Our second technical result now reads as follows:

\begin{lemma}\label{lem:deter-1}
$(i)$ If $H\in (\frac14,\frac12]$ and if $x:[0,T] \to E$ is a $H$-H{\"o}lder path, then for all $0\leq s\leq t\leq T$, it holds that 
$$\|x\, \1_{[s,t]}-x^{\Delta_{st}}\|_{1;H,T,E} \stackrel{|\Delta_{st}| \to 0}{\longrightarrow} 0$$
for any subdivision $\Delta_{st}=\{s=r_0<r_1<\cdots <r_n=t\}$ of $[s,t]$ whose mesh $|\Delta_{st}|$ tends to $0$, and where we have set  $x^{\Delta_{st}}_{u}:=\sum_{i=0}^{n-1} x_{r_i} \,   \1_{[r_i,r_{i+1})}(u)$.

\smallskip

\noindent
$(ii)$ If $H\in (\frac14,\frac12]$ and if $x:[0,T] \to E$ is a $H$-H{\"o}lder path, then for all $0\leq s\leq t\leq T$, one has 
$$\|\mathbf{x}^{s,t}-\mathbf{x}^{\Delta_{st}}\|_{2;H,T,E} \stackrel{|\Delta_{st}| \to 0}{\longrightarrow} 0$$
for any subdivision $\Delta_{st}=\{s=r_0<r_1<\cdots <r_n=t\}$ of $[s,t]$ whose mesh $|\Delta_{st}|$ tends to $0$, and where we have set $\mathbf{x}^{s,t}_{v,u}:=x_v \,  \1_{[s,t]}(v)\1_{[0,v)}(u)$ and $\mathbf{x}^{\Delta_{st}}_{v,u}:=\sum_{i=0}^{n-1} x_{r_i} \,   \1_{[r_i,r_{i+1})}(v)\1_{[0,r_i)}(u)$.
\end{lemma}

\begin{proof}
It is only a matter of standard fractional estimates (see e.g. \cite[Section 8]{alos-mazet-nualart}). For instance, if $[s,t]=[0,T]$, and picking $\varepsilon \in (0,2H-\frac12)$, we can use the $H$-H{\"o}lder regularity of $x$ to write, for any subdivision $\Delta=\{0=r_0<r_1<\cdots <r_n=T\}$ of $[0,T]$,
\begin{align*}
&\int_0^T du \bigg( \int_u^T dv \, |v-u|^{H-\frac32} \big\| x_v-x^{\Delta}_v-x_u+x^{\Delta}_u\|_{E} \bigg)^2 \\
&\leq c_{x,H,T} \bigg[\sum_{i=0}^{n-1} \int_{r_i}^{r_{i+1}} du \bigg( \int_u^{r_{i+1}} dv \, |v-u|^{2H-\frac32}\bigg)^2\\
&\hspace{2cm} +|\Delta|^{2\varepsilon} \sum_{i=0}^{n-1} \int_{r_i}^{r_{i+1}} du \bigg( \int_{r_{i+1}}^T dv \, |v-u|^{H-\frac32} \big\{ |v-u|^{H-\varepsilon}+|u-r_i|^{H-\varepsilon}\big\}\bigg)^2\bigg]\\
&\leq c_{x,H,T}\bigg[\sum_{i=0}^{n-1} \int_{r_i}^{r_{i+1}} du \, |r_{i+1}-u|^{4H-1}\\
&\hspace{1.5cm}+|\Delta|^{2\varepsilon} \sum_{i=0}^{n-1} \int_{r_i}^{r_{i+1}} du \, |T-u|^{4H-1-2\varepsilon}+|\Delta|^{2\varepsilon} \sum_{i=0}^{n-1} \int_{r_i}^{r_{i+1}} du\, |u-r_i|^{2H-2\varepsilon} |r_{i+1}-u|^{2H-1}\bigg]\\
&\leq c_{x,H,T}\bigg[|\Delta|^{4H-1}+|\Delta|^{2\varepsilon}+|\Delta|^{2\varepsilon} \sum_{i=0}^{n-1} |r_{i+1}-r_i|^{4H-2\varepsilon}\bigg] \ \leq \ c_{x,H,T} \Big[|\Delta|^{4H-1}+|\Delta|^{2\varepsilon}\Big] \, ,
\end{align*}
where the constant $c_{x,H,T}>0$ may of course change from one line to another.
\end{proof}

\

\begin{proof}[Proof of Proposition \ref{prop:skoro-integr}]

\

\smallskip

Combining inequalities (\ref{bound-ito}) and (\ref{identif-space}), it suffices to check that the paths $\bu:=P(X)\otimes Q(X)$, $D^X_1 \bu$ and $D^X_2 \bu$ meet the requirements of Lemma \ref{lem:deter-1}. 

\smallskip

The result for $\bu$ is an immediate consequence of the regularity of $X$, which is known to be $H$-H{\"o}lder with respect to the $L^\infty(\vp)$-norm (see e.g. \cite[Lemma 2.1]{deya-schott-3}). 

\smallskip

As for $D^X_1 \bu$ and $D^X_2 \bu$, observe for instance that
$$(D^X_1 \bu)_{v,u}=\mathbb{Y}_v \, \1_{[0,v)}(u) \quad \text{with} \quad \mathbb{Y}_v:=\partial P(X_v)\otimes Q(X_v) \, ,$$
and, with the notation of Lemma \ref{lem:deter-1} (point $(ii)$), $(D^X_1 \bu^{\Delta_{st}})_{v,u}=(D^X_1 \bu)^{\Delta_{st}}_{v,u}$. The convergence property
$$\big\| D^X_1 \bu^{\Delta_{st}} -(D^X_1\bu)^{s,t} \big\|_{2;T,H,\vp} \stackrel{|\Delta_{st}|\to 0}{\longrightarrow} 0$$
is thus again a consequence of the $H$-H{\"o}lder roughness of $X$ (for the $L^\infty(\vp)$-norm).
\end{proof}

Let us conclude Section \ref{subsec:skoro-fbm} with the following identification result: in brief, the Skorohod integral defined via Proposition \ref{prop:skoro-integr} merges with the It{\^o} integral (as defined in \cite[Corollary 3.1.2]{biane-speicher}) when $H=\frac12$, that is when working with the free Brownian motion. This is of course the (partial) NC counterpart of the classical Skorohod/It{\^o} identification result (see e.g. \cite[Proposition 1.3.11]{nualart-book}). The property can be stated as follows:

\begin{proposition}\label{prop:identif-ito}
When $H=\frac12$, one has, for all $0\leq s\leq t\leq T$,
$$\delta^X_{s,t}\big( P(X) \otimes Q(X)\big)=\int_s^t P(X_u) dX_u Q(X_u) \, ,$$
where the integral in the right-hand side is the (NC) It{\^o} integral defined in \cite[Corollary 3.1.2]{biane-speicher}.
\end{proposition}

\begin{proof}
When $H=\frac12$, it is a well-known fact that the space $\ch$ under consideration reduces to $\ch=L^2([0,T])$, and $\langle .,.\rangle_\ch=\langle .,.\rangle_{L^2([0,T])}$.

\smallskip

Now let us set $\bu_r:=P(X_r)\otimes Q(X_r)$ for every $r\in [0,T]$, and, for a given subdivision $\Delta_{st}:=\{s=r_0<\ldots <r_n=t\}$ with mesh $|\Delta_{st}|$ tending to $0$, let us define $\bu^{\Delta_{st}}$ along \eqref{bu-n}. 

\smallskip

We can first apply the very definition \eqref{defi-diver} of $\delta^X$ to write
\begin{align*}
\delta^X\big(\bu^{\Delta_{st}} \big)&=\sum_{i=0}^{n-1} \delta^X\big(\bu_{r_i} \, \1_{[r_i,r_{i+1})} \big)\\
&=\sum_{i=0}^{n-1} \Big\{P(X_{r_i}) \{X_{r_{i+1}}-X_{r_i}\} Q(X_{r_i})-(\id \times \vp \times \id)(\langle D^X\bu_{r_i}, \1_{[r_i,r_{i+1})} \rangle_{L^2([0,T])} \Big\} \, .
\end{align*}
Observe that in this situation, we have the explicit expression
$$D^X\bu_{r_i}=\big\{ \partial P(X_{r_i}) \otimes Q(X_{r_i})+P(X_{r_i})\otimes \partial Q(X_{r_i}) \big\} \1_{[0,r_{i})} \, ,$$
and therefore we simply end up with
$$\delta^X\big(\bu^{\Delta_{st}} \big)=\sum_{i=0}^{n-1} P(X_{r_i}) \{X_{r_{i+1}}-X_{r_i}\} Q(X_{r_i})=\int_s^t \bu^{\Delta_{st}}_r \sharp dX_r \, .$$
According to \cite[Corollary 3.1.2]{biane-speicher}, it remains us to check that
$$\int_s^t \big\| \bu_r-\bu^{\Delta_{st}}_r\big\|_{L^2(\vp \times \vp)}^2 \, dr \ \stackrel{|\Delta_{st}|\to 0}{\longrightarrow} 0 \, ,$$
but this property is of course a straightforward consequence of the $H$-H{\"o}lder regularity of $X$ (for the $L^\infty(\vp)$-norm), which achieves the proof of our assertion.

\end{proof}

\subsection{Pathwise integration with respect to the NC-fBm}\label{subsec:pathwise-integr}

\

\smallskip

As we announced it in the introduction, our objective in the next section will be to compare the previous Skorohod approach with the pathwise constructions of \cite{deya-schott-3}.

\smallskip

For the reader's convenience, we propose to briefly recall how pathwise integrals with respect to the NC-fBm can be defined. Therefore, let $\{X_t, \, t\in [0,T]\}$ be NC-fBm of Hurst index $H>\frac13$, in a given NC probability space $(\ca,\vp)$. For the sake of conciseness, we will henceforth denote the increments of $X$ as $X_{s,t}:=X_t-X_s$.

\smallskip

Just as in the classical commutative case, we need to separate the two cases $H>\frac12$ and $H\leq \frac12$ in the presentation of the results.

\begin{proposition}[Proposition 2.5 in \cite{deya-schott-3}]\label{prop:int-young}
Assume that $H>\frac12$. Then, for all polynomials $P,Q$, all times $0\leq s\leq t \leq T$ and every subdivision $\Delta_{st} = \{s=r_0<r_1 <\ldots<r_\ell=t\}$ of $[s,t]$ with mesh $|\Delta_{st}|$ tending to $0$, the Riemann sum
\begin{equation}\label{riemann}
\sum_{i=0}^{\ell-1}  P(X_{r_i}) X_{r_i,r_{i+1}}Q(X_{r_i})
\end{equation}
converges in $\ca$ (i.e. for the $L^\infty(\vp)$-norm) as $|\Delta_{st}| \to 0$. We call the limit the \emph{Young integral} of $P(X)\otimes Q(X)$ with respect to $X$ on $[s,t]$, and we denote it by $\int_s^t P(X_u) d X_u Q(X_u)$ or $\int_s^t (P(X_u)\otimes Q(X_u)) \sharp dX_u$ (following the notation in \eqref{sharp-notation}). Besides, one has
\begin{equation}\label{ito-young}
 P(X_t)-P(X_s)=\int_s^t \partial P(X_u) \sharp dX_u \ .
\end{equation}
\end{proposition}

\smallskip

In order to go one step further and handle the so-called \emph{rough case}, that is the situation where $H\in (\frac13,\frac12]$, let us introduce, along the ideas of \cite{deya-schott-3}, the approximation $(X^{(n)})_{n\geq 0}$ of $X$ given by 
\begin{equation}\label{def-x-n-appro}
X^{(0)}_t=tX_1 \quad , \quad X^{(n)}_t:=X_{t^n_i}+2^n (t-t^n_i) X_{t_i^n,t^n_{i+1}} \quad \text{for} \ n\geq 1 \ \text{and} \ t\in [t^n_i,t^n_{i+1}]\ ,
\end{equation}
where $(t_i^n)$ stands for the dyadic partition of $[0,T]$, that is $t_i^n:=\frac{i T}{2^n}$, $i=0,\ldots,2^n$. Then define the sequence of approximated \emph{product L{\'e}vy areas} by the formula: for all $n\geq 0$, 
\begin{equation}\label{levy-area-appr}
\mathbb{X}^{2,n}_{s,t}[U]:=\int_s^t X^{(n)}_{s,u} U dX^{(n)}_u \ ,\quad 0\leq s\leq t\leq T\ , \ U\in \ca \, ,
\end{equation}
where the integral is here interpreted as a classical Lebesgue integral. In other words, if $s\in [t^n_k,t^n_{k+1}]$ and $t\in [t^n_\ell,t^n_{\ell+1}]$,
\begin{align*}
&\mathbb{X}^{2,n}_{s,t}[U]:=\\
&2^n\bigg(\int_s^{t_{k+1}^n} X^{(n)}_{s,u} \, du\bigg) UX_{t_k^n,t^n_{k+1}}  +\sum_{i=k+1}^{\ell-1} 2^n\bigg(\int_{t_i^n}^{t_{i+1}^n} X^{(n)}_{s,u}\, du\bigg) U X_{t_i^n,t^n_{i+1}} +2^n\bigg(\int_{t_\ell^n}^t X^{(n)}_{s,u}\, du\bigg)  U X_{t_\ell^n,t^n_{\ell+1}}\, .
\end{align*}

\begin{proposition}[Proposition 2.8 and 2.9 in \cite{deya-schott-3}]\label{prop:rough-int}
Assume that $H\in (\frac13,\frac12]$. Then:

\smallskip

\noindent
$(i)$ For all $0\leq s\leq t\leq T$ and $U\in \ca_s$, the sequence $\mathbb{X}_{st}^{2,n}[U]$ converges in $L^\infty(\vp)$ as $n\to \infty$. We denote its limit by $\mathbb{X}_{st}^{2}[U]$, and then set, for all $U,V \in \ca_s$, 
$$(U\otimes V)\sharp \mathbb{X}^2_{st}:=U\mathbb{X}^2_{st}[V]\quad , \quad \mathbb{X}^{2,\ast}_{st}\sharp (U\otimes V):=\mathbb{X}^2_{st}[U^\ast]^\ast V\, .$$


\noindent
$(ii)$ For all polynomials $P,Q$, all $0\leq s\leq t \leq T$ and every subdivision $\Delta_{st} = \{s=r_0<r_1 <\ldots<r_\ell=t\}$ of $[s,t]$ with mesh $|\Delta_{st}|$ tending to $0$, the corrected Riemann sum
\begin{equation}\label{correc-riem-sum}
\sum_{i=0}^{\ell-1} \Big\{ P(X_{r_i})X_{r_i,r_{i+1}}Q(X_{r_i})+(\partial P (X_{r_i}) \sharp \mathbb{X}^2_{r_i,r_{i+1}}) Q(X_{r_i})+P(X_{r_i})(\mathbb{X}^{2,\ast}_{r_i,r_{i+1}} \sharp \partial Q(X_{r_i}))\Big\}
\end{equation}
converges in $\ca$ (i.e. for the $L^\infty(\vp)$-norm) as $|\Delta_{st}| \to 0$. We call its limit the \emph{rough integral} of $P(X)\otimes Q(X)$ with respect to $(X,\mathbb{X}^2)$ (on $[s,t]$), and denote it by $\int_s^t P(X_u) (\circ \, d \mathbb{X}_u) Q(X_u)$ or $\int_s^t (P(X_u)\otimes Q(X_u)) \sharp (\circ \, d \mathbb{X}_u)$.

\smallskip

\noindent
$(iii)$ For every polynomial $P$ and all $0\leq s\leq t \leq T$, it holds that
\begin{equation}\label{ito-rough}
P(X_t)-P(X_s)=\int_s^t \partial P(X_u) \sharp (\circ \, d \mathbb{X}_u) \ .
\end{equation}
\end{proposition}

\

\begin{remark}
The above respective definitions of the Young and rough integrals are known to be consistent with each other, in the sense that if $H\in (\frac12,1)$, one has
$$\int_s^t P(X_u) (\circ \, d \mathbb{X}_u) Q(X_u)=\int_s^t P(X_u) d X_u Q(X_u) \, .$$
This is due to the \enquote{$(2H-\varepsilon)$-regularity} of $\mathbb{X}^{2}$ (see \cite[Proposition 2.8]{deya-schott-3} for a precise statement of this regularity property).
\end{remark}

\begin{remark}\label{rk:strato-int}
The above Young and rough integrals can both be seen as natural fractional extension of the Stratonovich integral. Indeed, let us recall that these integrals can also be obtained (in a less \enquote{intrinsic} way) as the limit in $\ca$ of the sequence of classical Lebesgue integrals
$$\int_s^t P(X^{(n)}_u) \, dX^{(n)}_u\, Q(X^{(n)}_u) \, .$$ 
When $H=\frac12$, this limit corresponds precisely to the Stratonovich integral (see e.g. \cite[Proposition 5.5]{deya-schott}), just as in the classical commutative framework.
\end{remark}

\section{An It{\^o}-Stratonovich formula}\label{sec:ito-strato-correc}

We are now in a position to state the main result of our study, namely a specific comparison formula between the Skorohod and the pathwise integrals with respect to the NC-fBm (as defined in Sections \ref{subsec:skoro-fbm} and \ref{subsec:pathwise-integr}, respectively). With the identifications of Proposition \ref{prop:identif-ito} and Remark \ref{rk:strato-int} in mind, such a property can legitimately be regarded as a (NC) \emph{It{\^o}-Stratonovich correction formula}.

\begin{theorem}\label{main-theo}
Fix $T>0$ and let $\{X_t, \, t\in [0,T]\}$ be a NC-fBm of Hurst index $H\in (\frac13,1)$, in a given NC probability space $(\ca,\vp)$. Then, for all $0\leq s\leq t\leq T$ and all polynomials $P,Q$, it holds that
\begin{equation}\label{correc-ito-strato}
\begin{split}
&\int_s^t P(X_u) (\circ \, d\mathbb{X}_u) Q(X_u)= \delta^X_{s,t}\big( P(X)\otimes Q(X)\big)\\
&\hspace{2cm}+H \int_s^t du \, u^{2H-1} \big( \emph{\id} \times \vp \times \emph{\id}\big) \big[\partial P(X_u) \otimes Q(X_u)+P(X_u)\otimes \partial Q (X_u) \big] \, .
\end{split}
\end{equation}
\end{theorem}

\

Combining the differential rules \eqref{ito-young} and \eqref{ito-rough} with the decomposition \eqref{correc-ito-strato}, we immediately deduce the following It{\^o}-type formula for the Skorohod integral: 

\begin{corollary}\label{coro:ito-form}
Fix $T>0$ and let $\{X_t, \, t\in [0,T]\}$ be a NC-fBm of Hurst index $H\in (\frac13,1)$, in a given NC probability space $(\ca,\vp)$. Then, for all $0\leq s\leq t\leq T$ and for every polynomial $P$, it holds that
\begin{equation}\label{ito-form}
P(X_t)-P(X_s)=\delta^X_{s,t}\big( \partial P(X) \big)+H \int_s^t du \, u^{2H-1} \big( \emph{\id} \times \vp \times \emph{\id} \big) \big(\partial^2 P(X_u)\big) \, ,
\end{equation}
where the notation $\partial^2 P$ (for the second-order tensor derivative) has been introduced in \eqref{d-deux-p}.
\end{corollary}

Formula \eqref{ito-form} is nothing but the extension, to every $H\in (\frac13,1)$, of the (NC) It{\^o} formula exhibited in \cite[Section 5.1]{deya-schott} for the particular free Brownian case $H=\frac12$. In the same way, Theorem \ref{main-theo} corresponds to the fractional extension of the result of \cite[Proposition 5.6]{deya-schott}, at least when working with polynomial integrands.

\smallskip

As for the comparison with the classical commutative framework, observe for instance that Corollary \ref{coro:ito-form} can be seen as the NC counterpart of the result of \cite[Theorem 5.2.2]{nualart-book}, while \eqref{correc-ito-strato} is morally the NC version of identity \eqref{ito-strato-commuta} (see also \cite{cheridito-nualart} for similar fractional formulas). A remarkable feature to be noticed here is the specific involvement of $\vp$ in the \enquote{trace} terms of \eqref{correc-ito-strato} and \eqref{ito-form}. Of course, such an involvement could not be guessed from the corresponding commutative formula (it could rather be guessed from the definition \eqref{defi-diver} of the NC divergence operator). 

\

\emph{In order to prove Theorem \ref{main-theo}, and in accordance with the splitting used in Section \ref{subsec:pathwise-integr}, we will need to treat the two situations $H>\frac12$ and $H\leq \frac12$ separately.}

\

\begin{remark}
As a particular consequence of formula \eqref{correc-ito-strato}, we can assert that the Skorohod integral $\delta^X_{s,t}\big( P(X)\otimes Q(X)\big)$ actually belongs to $\ca$ ($=L^\infty(\vp)$), and not only $L^2(\vp)$, because 
\small
$$\int_s^t P(X_u) (\circ \, d\mathbb{X}_u) Q(X_u)\quad \text{and} \quad \int_s^t du \, u^{2H-1} \big( \id \times \vp \times \id \big) \big[\partial P(X_u) \otimes Q(X_u)+P(X_u)\otimes \partial Q (X_u) \big]$$
\normalsize
both belong to $\ca$. The fact that $\delta^X_{s,t}\big( P(X)\otimes Q(X)\big)\in \ca$ was not obvious at first sight, if we only refer to Definition \ref{defi:skoro-int} and Proposition \ref{prop:skoro-integr}.
\end{remark}

\

\subsection{Proof of Theorem \ref{main-theo} in the Young case $H>\frac12$}\label{subsec:proof-young}

\

\smallskip

Let us set $\bu:=P(X) \otimes Q(X)$. Then, using the very definition \eqref{defi-diver} of $\delta^X$, as well as the first identity in \eqref{deriv-p-x-h}, we get, for any subdivision $\Delta_{st}:=\{s=r_0<r_1<\ldots <r_\ell=t\}$,
\begin{align}
&\delta^X \big( \bu^{\Delta_{st}}\big)=\sum_{i=0}^{\ell-1} \delta^X \big( P(X_{r_i}) \otimes Q(X_{r_i}) \, \1_{[r_i,r_{i+1})} \big)\nonumber\\
&=\sum_{i=0}^{\ell-1} P(X_{r_i}) X_{r_i,r_{i+1}} Q(X_{r_i})\nonumber\\
&\hspace{0.5cm}-\sum_{i=0}^{\ell-1} \big( \id \times \vp \times \id \big) \big[\partial P(X_{r_i}) \otimes Q(X_{r_i})+P(X_{r_i})\otimes \partial Q (X_{r_i}) \big]\langle \1_{[0,r_i)},\1_{[r_i,r_{i+1})} \rangle_{\ch} \, .\label{decompo-sum-young}
\end{align}
Now, setting $f_H(t):=\frac12 |t|^{2H}$, it is easy to check that 
$$\langle \1_{[0,r_i)},\1_{[r_i,r_{i+1})} \rangle_{\ch}=\big\{f_H(r_{i+1})-f_H(r_i)\big\}-f_H(r_{i+1}-r_i)\, ,$$
and since $H>\frac12$, one has here
$$f_H(r_{i+1}-r_i)=\frac12 |r_{i+1}-r_i|^{2H} \leq \frac12 |r_{i+1}-r_i| |\Delta_{st}|^{2H-1} \, .$$
As a result,
\begin{align*}
&\sum_{i=0}^{\ell-1} \big( \id \times \vp \times \id \big) \big[\partial P(X_{r_i}) \otimes Q(X_{r_i})+P(X_{r_i})\otimes \partial Q (X_{r_i}) \big]\langle \1_{[0,r_i)},\1_{[r_i,r_{i+1})} \rangle_{\ch}\\
&=\sum_{i=0}^{\ell-1} \big( \id \times \vp \times \id \big) \big[\partial P(X_{r_i}) \otimes Q(X_{r_i})+P(X_{r_i})\otimes \partial Q (X_{r_i}) \big]\big\{f_H(r_{i+1})-f_H(r_i)\big\}\\
&\hspace{0.5cm}-\sum_{i=0}^{\ell-1} \big( \id \times \vp \times \id \big) \big[\partial P(X_{r_i}) \otimes Q(X_{r_i})+P(X_{r_i})\otimes \partial Q (X_{r_i}) \big]f_H(r_{i+1}-r_i)\\
\end{align*}
with
\begin{align*}
&\bigg\|\sum_{i=0}^{\ell-1} \big( \id \times \vp \times \id \big) \big[\partial P(X_{r_i}) \otimes Q(X_{r_i})+P(X_{r_i})\otimes \partial Q (X_{r_i}) \big]f_H(r_{i+1}-r_i)\bigg\|_{L^\infty(\vp)}\\
&\leq \frac12 |\Delta_{st}|^{2H-1} |t-s|\sup_{u\in [s,t]} \big\|\big( \id \times \vp \times \id \big) \big[\partial P(X_{u}) \otimes Q(X_{u})+P(X_{u})\otimes \partial Q (X_{u}) \big]\big\|_{L^\infty(\vp)}\stackrel{|\Delta_{st}|\to 0}{\longrightarrow} 0\, .
\end{align*}
This shows that
\begin{align*}
&\sum_{i=0}^{\ell-1} \big( \id \times \vp \times \id \big) \big[\partial P(X_{r_i}) \otimes Q(X_{r_i})+P(X_{r_i})\otimes \partial Q (X_{r_i}) \big]\langle \1_{[0,r_i)},\1_{[r_i,r_{i+1})} \rangle_{\ch}\\
&\hspace{0.5cm}\stackrel{|\Delta_{st}|\to 0}{\longrightarrow} \int_s^t du \, f_H'(u) \big( \id \times \vp \times \id \big) \big[\partial P(X_u) \otimes Q(X_u)+P(X_u)\otimes \partial Q (X_u) \big] \quad \text{in} \ L^\infty(\vp) \, .
\end{align*}
Going back to \eqref{decompo-sum-young} and letting $|\Delta_{st}|$ tend to $0$, we can use the result of Proposition \ref{prop:int-young} to deduce both the convergence of $\delta^X \big( \bu^{\Delta_{st}}\big)$ in $L^\infty(\vp)$ (with limit $\delta^X_{s,t} \big( \bu\big)$) and the decomposition \eqref{correc-ito-strato}, which completes the proof of Theorem \ref{main-theo} in the case $H>\frac12$.

\smallskip


\subsection{Proof of Theorem \ref{main-theo} in the rough case $H\in (\frac13,\frac12]$}

\

\smallskip

In this situation, we additionally have to consider the second-order elements involved in \eqref{correc-riem-sum}, that is the L{\'e}vy-area terms. The cornerstone toward Theorem \ref{main-theo} will thus be the following approximation property for the sum of these terms:

\begin{proposition}\label{prop:rough-case}
Assume that $H\in (\frac13,\frac12]$, and for every $n\geq 1$, consider the dyadic partition $t_i^n:=\frac{iT}{2^n}$, $0\leq i\leq 2^n$. Then, for all $0\leq s<t\leq T$ and $d,d_1,d_2\geq 0$, one has
\begin{equation}\label{conv-loc-lev}
\sum_{i=0}^{2^n-1} X_{t_i^n}^{d_1}\Big\{\mathbb{X}^{2}_{t_i^n,t^n_{i+1}} \big[ X_{t^n_i}^d\big] -\frac12 \vp\big( X_{t_i^n}^d\big) |t^n_{i+1}-t^n_i|^{2H}\Big\}X_{t_i^n}^{d_2} \stackrel{n\to \infty}{\longrightarrow} 0 \quad \text{in} \ L^{2}(\vp) \, .
\end{equation}
\end{proposition}

\smallskip

For the sake of clarity, we have postponed the (technical) proof of this assertion to the subsequent Appendix \ref{sec:proof-approx-sum-levy}. Let us now see how the property will lead to our main statement.

\begin{proof}[Proof of Theorem \ref{main-theo} when $\frac13<H\leq \frac12$]

\

\smallskip

We will focus on the situation where $s=0$ and $t=T$, but it is easy to see that the case of a general interval $[s,t]$ could be handled along similar arguments (we leave these modifications, and especially the adaptation of Proposition \ref{prop:rough-case}, as an exercise to the reader).  

\smallskip
 
Just as in Section \ref{subsec:proof-young}, let us set $\bu:=P(X) \otimes Q(X)$. Also, let us define $\bu^n:=\bu^{\Delta^n_{0T}}$, where $\Delta^n_{0T}=(t_i^n)$ refers to the dyadic approximation of $[0,T]$, that is $t_i^n:=\frac{iT}{2^n}$. 

\smallskip

Now observe that
\begin{align}
&\delta^X \big( \bu^n\big)=\sum_{i=0}^{2^n-1} \delta^X \big( P(X_{t_i^n}) \otimes Q(X_{t_i^n}) \, \1_{[t_i^n,t_{i+1}^n)} \big)\nonumber\\
&=\sum_{i=0}^{2^n-1} P(X_{t_i^n}) X_{t_i^n,t_{i+1}^n} Q(X_{t_i^n})\nonumber\\
&\hspace{0.5cm}-\sum_{i=0}^{2^n-1} \big( \id \times \vp \times \id \big) \big[\partial P(X_{t_i^n}) \otimes Q(X_{t_i^n})+P(X_{t_i^n})\otimes \partial Q (X_{t_i^n}) \big]\langle \1_{[0,t_i^n)},\1_{[t^n_i,t^n_{i+1})} \rangle_{\ch} \nonumber \\
&=\sum_{i=0}^{2^n-1} \Big[ P(X_{t_i^n}) X_{t_i^n,t_{i+1}^n} Q(X_{t_i^n}) +\big( \partial P (X_{t_i^n}) \sharp \mathbb{X}^2_{t_i^n,t_{i+1}^n} \big) Q(X_{t_i^n})+P(X_{t_i^n}) \big( \mathbb{X}^{2,\ast}_{t_i^n,t_{i+1}^n} \sharp \partial Q (X_{t_i^n}) \big) \Big]\nonumber\\
&-\frac12 \sum_{i=0}^{2^n-1} \big( \id \times \vp \times \id \big) \big[\partial P(X_{t_i^n}) \otimes Q(X_{t_i^n})+P(X_{t_i^n})\otimes \partial Q (X_{t_i^n}) \big]\big\{|t_{i+1}^n|^{2H}-|t^n_i|^{2H} \big\}-R_n \, ,\label{decompo-skorohod}
\end{align}
where we have set $R_n:=R_n^1+R_n^2$, with
$$
R^1_n:=\sum_{i=0}^{2^n-1} \Big[ \big( \partial P (X_{t_i^n}) \sharp \mathbb{X}^2_{t_i^n,t_{i+1}^n} \big) Q(X_{t_i^n})-\frac12 \big( \id \times \vp \times \id \big) \big[\partial P(X_{t_i^n}) \otimes Q(X_{t_i^n})\big] |t_{i+1}^n-t_i^n|^{2H} \Big]
$$
and
$$
R_n^2:=\sum_{i=0}^{2^n-1} \Big[ P(X_{t_i^n}) \big( \mathbb{X}^{2,\ast}_{t_i^n,t_{i+1}^n} \sharp \partial Q (X_{t_i^n}) \big)-\frac12 \big( \id \times \vp \times \id \big) \big[P(X_{t_i^n})\otimes \partial Q (X_{t_i^n}) \big]   |t_{i+1}^n-t_i^n|^{2H}\Big] \, .
$$
Assuming for simplicity that $P(x):=x^p$ and $Q(x):=x^q$, we can write $R^1_n$ as
$$R_n^1=\sum_{k=0}^{p-1} \sum_{i=0}^{2^n-1}X_{t_i^n}^k \Big[ \big(   \mathbb{X}^2_{t_i^n,t_{i+1}^n}\big[X_{t_i^n}^{p-1-k}\big]  -\frac12 \vp\big(X_{t_i^n}^{p-1-k}\big) |t_{i+1}^n-t_i^n|^{2H} \Big]X_{t_i^n}^q \, ,$$
and thus, by Proposition \ref{prop:rough-case}, we can assert that $R^1_n \to 0$ in $L^2(\vp)$ as $n\to \infty$.

\smallskip

The same argument clearly applies to $R^2_n$.

\smallskip

Since $\bu$ is known to be Skorohod integrable (by Proposition \ref{prop:skoro-integr}), we can now let $n$ tend to infinity in both sides of \eqref{decompo-skorohod} (considering the $L^2(\vp)$-norm), which, together with the convergence result of Proposition \ref{prop:rough-int} (item $(ii)$), yields the desired decomposition \eqref{correc-ito-strato}.

\end{proof}


\appendix

\section{Proof of Proposition \ref{prop:rough-case}}\label{sec:proof-approx-sum-levy}

According to \cite[Proposition 2.8]{deya-schott-3}, we know that the product L{\'e}vy area $\mathbb{X}^{2}$ is $(2H-\varepsilon)$-H{\"o}lder regular, in the following specific sense: for every polynomial expression $P$ (say of $r$ arguments), there exists a constant $c_P>0$ such that for all $0\leq s<t\leq T$ and  $u_1,\ldots,u_r\in [0,s]$, 
$$\big\| \mathbb{X}^{2}_{s,t}\big[P(X_{u_1},\ldots,X_{u_r})\big] \big\|_{L^\infty(\vp)} \leq c_P \, |t-s|^{2H-\varepsilon} \, .$$
Unfortunately, such a regularity property is clearly not sufficient to deduce the desired local approximation \eqref{conv-loc-lev}. We will thus need to go deeper into the properties of $\mathbb{X}^{2}$, and accordingly into the properties of $X$.

\smallskip

In fact, just as in the commutative situation (see e.g. the proof of \cite[Proposition 5.1]{cass-lim}), we will be led to exploit, at some point of our analysis, the 2D-regularity properties of the fractional covariance $R_H$ (defined in \eqref{cova-NC-fBm}). In order to state these properties, let us recall first that for any function $f:[0,T]^2 \to \R$, and setting
$$f\begin{pmatrix} s&t\\u &v \end{pmatrix}:=f(t,v)-f(t,u)-f(s,v)+f(s,u) \, ,$$
we define, for every $\rho \geq 1$, the (2D) $\rho$-variation of $f$ on $I\times J \subset [0,T]^2$ as
$$\|f\|_{\rho-var;I\times J}:= \sup_{(s_i)\in \cp_I,(t_j)\in \cp_J}\bigg( \sum_{i,j} \bigg| f\begin{pmatrix} s_i & s_{i+1}\\ t_j & t_{j+1} \end{pmatrix}\bigg|^\rho\bigg)^{\frac{1}{\rho}}$$
If $\|f\|_{\rho-var;[0,T]^2} < \infty$, then we say that $f$ is of finite (2D) $\rho$-variation on $[0,T]^2$, and in this case, we know that the map
$$\omega_{f,\rho}(I\times J):=\|f\|_{\rho-var;I\times J}^\rho \ , \quad I,J\subset [0,T] \, , $$
defines a 2D-control on $[0,T]^2$, that is: for all rectangles $R_1,R_2,R \subset [0,T]^2$ such that $R_1\cap R_2=\emptyset$ and $R_1\cup R_2\subset R$,
$$\omega_{f,\rho}(R_1)+\omega_{f,\rho}(R_2)\leq \omega_{f,\rho}(R) \ .$$

\smallskip

With these preliminaries in mind, the specific regularity properties that we shall use in the sequel can be stated as follows (see e.g. \cite[Proposition 14]{friz-victoir-gauss} for a proof of these results):
\begin{proposition}
Let $R_H$ stand for the fractional covariance of Hurst index $H$ (see \eqref{cova-NC-fBm}), and assume that $H\in(\frac13,\frac12]$. Then it holds that
\begin{equation}\label{2d-var-r}
\big\|R_H\big\|_{\frac{1}{2H}-var;[0,T]^2} \ < \ \infty \, ,
\end{equation}
and for all $0\leq s<t\leq T$ and $u\in [0,T]$, one has
\begin{equation}\label{bou-r-h-2}
\big| R_H(t,u)-R_H(s,u)\big| \leq c_H |t-s|^{2H} \, .
\end{equation}
\end{proposition}

\

\large
\begin{center}
\textbf{Proof of Proposition \ref{prop:rough-case}.}
\end{center}
\normalsize

\smallskip

Throughout the proof, we will use the notation $A\lesssim B$ in order to signify that there exists an irrelevant constant $c$ such that $A\leq c B$.

\

For each $n\geq 1$, let us first rewrite the quantity under consideration, that is
$$\sum_{i=0}^{2^n-1} X_{t_i^n}^{d_1}\Big\{\mathbb{X}^{2}_{t_i^n,t^n_{i+1}} \big[ X_{t^n_i}^d\big] -\frac12 \vp\big( X_{t_i^n}^d\big) \vp\big( X_{t_i^n,t_{i+1}^n}^2\big)\Big\}X_{t_i^n}^{d_2} \, ,$$
as the limit in $L^\infty(\vp)$, and as $N\to \infty$, of the sequence
\begin{equation}\label{appro-x-2}
\sum_{i=0}^{2^n-1} X_{t_i^n}^{d_1}\Big\{\mathbb{X}^{2,N}_{t_i^n,t^n_{i+1}} \big[ X_{t^n_i}^d\big] -\frac12 \vp\big( X_{t_i^n}^d\big) \vp\big( X_{t_i^n,t_{i+1}^n}^2\big)\Big\}X_{t_i^n}^{d_2} \, ,
\end{equation}
where $\mathbb{X}^{2,N}$ stands for the approximation of $\mathbb{X}^{2}$ given by \eqref{levy-area-appr}. Thus, in the sequel, we will search for a bound of \eqref{appro-x-2} (in $L^2(\vp)$) that is uniform in $N$ and converges to $0$ as $n\to\infty$. To this end, let us consider, for every $N\geq n$, the telescopic decomposition 
$$\mathbb{X}^{2,N}_{t_i^n,t^n_{i+1}} \big[ X_{t^n_i}^d\big]=\mathbb{X}^{2,n}_{t_i^n,t^n_{i+1}} \big[ X_{t^n_i}^d\big]+\sum_{m=n}^{N-1} \Big\{\mathbb{X}^{2,m+1}_{t_i^n,t^n_{i+1}} \big[ X_{t^n_i}^d\big]-\mathbb{X}^{2,m}_{t_i^n,t^n_{i+1}} \big[ X_{t^n_i}^d\big]\Big\} \, ,$$
which allows us to expand the quantity in \eqref{appro-x-2} as
\begin{align}
&\sum_{i=0}^{2^n-1} X_{t_i^n}^{d_1}\Big\{\mathbb{X}^{2,n}_{t_i^n,t^n_{i+1}} \big[ X_{t^n_i}^d\big] -\frac12 \vp\big( X_{t_i^n}^d\big) \vp\big( X_{t_i^n,t_{i+1}^n}^2\big)\Big\}X_{t_i^n}^{d_2}\nonumber\\
&\hspace{1cm}+\sum_{i=0}^{2^n-1} \sum_{m=n}^{N-1}X_{t_i^n}^{d_1} \Big\{\mathbb{X}^{2,m+1}_{t_i^n,t^n_{i+1}} \big[ X_{t^n_i}^d\big]-\mathbb{X}^{2,m}_{t_i^n,t^n_{i+1}} \big[ X_{t^n_i}^d\big]\Big\}X_{t_i^n}^{d_2} \ =: \ I_n+II_{N,n} \, .\label{decompo-i-ii}
\end{align}
For the sake of clarity, let us set from now on 
\begin{equation}\label{y-m-j}
Y^{(m)}_j:=X_{t_j^{m+1},t_{j+1}^{m+1}} \quad \text{for all} \ m\geq 1 \ \text{and}\ j=0,\ldots,2^{m+1}-1\, .
\end{equation}

\subsection{Bound for $I_n$} The whole point here is that $\mathbb{X}^{2,n}_{t_i^n,t^n_{i+1}} \big[ X_{t^n_i}^d\big]$ actually reduces to
$$\mathbb{X}^{2,n}_{t_i^n,t^n_{i+1}} \big[ X_{t^n_i}^d\big]=\frac12 Y^{(n-1)}_i X_{t_i^n}^d Y^{(n-1)}_i\, ,$$
and so
\begin{align}
\big\|I_n\big\|_{L^2(\vp)}^2&=\frac14\sum_{i_1,i_2=0}^{2^n-1} \vp\bigg( \Big\{ X_{t_{i_1}^n}^{d_1} \Big[ Y^{(n-1)}_{i_1} X_{t_{i_1}^n}^d Y^{(n-1)}_{i_1}-\vp\big( X_{t_{i_1}^n}^d \big) \vp\big( \big( Y^{(n-1)}_{i_1}\big)^2\big) \Big] X_{t_{i_1}^n}^{d_2} \Big\}\nonumber\\
&\hspace{4cm} \cdot \Big\{ X_{t_{i_2}^n}^{d_2} \Big[ Y^{(n-1)}_{i_2} X_{t_{i_2}^n}^d Y^{(n-1)}_{i_2}-\vp\big( X_{t_{i_2}^n}^d \big) \vp\big( \big( Y^{(n-1)}_{i_2}\big)^2\big) \Big] X_{t_{i_2}^n}^{d_1} \Big\}\bigg)\nonumber\\
&=\frac14 \sum_{i_1,i_2=0}^{2^n-1} \bigg[ \vp \Big( X_{t_{i_1}}^{d_1} Y_{i_1} X_{t_{i_1}}^d Y_{i_1} X_{t_{i_1}}^{d_2} X_{t_{i_2}}^{d_2} Y_{i_2} X_{t_{i_2}}^d Y_{i_2} X_{t_{i_2}}^{d_1} \Big)\nonumber\\
&\hspace{2cm}-\vp \Big( X_{t_{i_1}}^{d_1} Y_{i_1} X_{t_{i_1}}^d Y_{i_1} X_{t_{i_1}}^{d_2} X_{t_{i_2}}^{d_2} X_{t_{i_2}}^{d_1} \Big)\vp\Big( X_{t_{i_2}}^d \Big) \vp\Big( Y_{i_2}^2\Big)\nonumber\\
&\hspace{2cm}-\vp \Big( X_{t_{i_1}}^{d_1}  X_{t_{i_1}}^{d_2} X_{t_{i_2}}^{d_2} Y_{i_2} X_{t_{i_2}}^d Y_{i_2} X_{t_{i_2}}^{d_1} \Big)\vp\Big( Y_{i_1}^2\Big) \vp\Big(X_{t_{i_1}}^d\Big)\nonumber\\
&\hspace{2cm}+\vp \Big( X_{t_{i_1}}^{d_1} X_{t_{i_1}}^{d_2} X_{t_{i_2}}^{d_2} X_{t_{i_2}}^{d_1} \Big)\vp\Big(X_{t_{i_1}}^d\Big) \vp\Big(X_{t_{i_2}}^d\Big)\vp\Big( Y_{i_1}^2\Big) \vp\Big( Y_{i_2}^2\Big)\bigg] \ ,\label{exp-i-n}
\end{align}
where, for more clarity, we have deliberately omitted the dependence on $n$ in the latter sum (in particular, each $Y_i$ must be read as $Y^{(n-1)}_i$).

\smallskip

Let us focus on the first term of this sum: according to the NC Wick formula (\ref{form-wick}), and setting $D:=2(d+d_1+d_2)+4$, this term can be expanded as
\begin{align}
&\vp \Big( X_{t_{i_1}}^{d_1} Y_{i_1} X_{t_{i_1}}^d Y_{i_1} X_{t_{i_1}}^{d_2} X_{t_{i_2}}^{d_2} Y_{i_2} X_{t_{i_2}}^d Y_{i_2} X_{t_{i_2}}^{d_1} \Big)\nonumber\\
&=\sum_{\pi\in NC_2(D)} \vp_{\pi}\Big(\underbrace{X_{t_{i_1}},\ldots,X_{t_{i_1}}}_{d_1 \, \text{occur.}},Y_{i_1},\underbrace{X_{t_{i_1}},\ldots,X_{t_{i_1}}}_{d \, \text{occur.}},Y_{i_1},\underbrace{X_{t_{i_1}},\ldots,X_{t_{i_1}}}_{d_2 \, \text{occur.}},\nonumber\\
&\hspace{4cm}\underbrace{X_{t_{i_2}},\ldots,X_{t_{i_2}}}_{d_2 \, \text{occur.}},Y_{i_2},\underbrace{X_{t_{i_2}},\ldots,X_{t_{i_2}}}_{d \, \text{occur.}},Y_{i_2},\underbrace{X_{t_{i_2}},\ldots,X_{t_{i_2}}}_{d_1 \, \text{occur.}} \Big)\label{ka-pi}\\
&=\sum_{\substack{\pi\in NC_2(D)\\(d_1+1,d_1+d+2) \notin \pi\\(D-(d_1+d+1),D-d_1)\notin \pi}} \vp_{\pi}\Big(\ldots\Big)+\sum_{\substack{\pi\in NC_2(D)\\(d_1+1,d_1+d+2) \in \pi}} \vp_{\pi}\Big(\ldots\Big)\nonumber\\
&\hspace{3cm}+\sum_{\substack{\pi\in NC_2(D)\\(D-(d_1+d+1),D-d_1)\in \pi}} \vp_{\pi}\Big(\ldots\Big)-\sum_{\substack{\pi\in NC_2(D)\\(d_1+d+1,d_1+d+2) \in \pi\\(D-(d_1+d+1),D-d_1)\in \pi}} \vp_{\pi}\Big(\ldots\Big) \ ,\label{sum-ka-pi}
\end{align}
noting that the pair $(d_1+1,d_1+d+2)$, resp. $(D-(d_1+d+1),D-d_1)$, is the one that \enquote{connects} the two occurrences of the variable $Y_{i_1}$, resp. $Y_{i_2}$, in the quantity $\vp_\pi\big(\ldots\big)$ under consideration here. Using the NC Wick formula again, it is easy to see that the last three sums actually correspond to

$$\sum_{\substack{\pi\in NC_2(D)\\(d_1+1,d_1+d+2) \in \pi}} \vp_{\pi}\Big(\ldots\Big)=\vp \Big( X_{t_{i_1}}^{d_1}  X_{t_{i_1}}^{d_2} X_{t_{i_2}}^{d_2} Y_{i_2} X_{t_{i_2}}^d Y_{i_2} X_{t_{i_2}}^{d_1} \Big)\vp\Big( Y_{i_1}^2\Big) \vp\Big(X_{t_{i_1}}^d\Big) \ ,$$
$$\sum_{\substack{\pi\in NC_2(D)\\(D-(d_1+d+1),D-d_1)\in \pi}} \vp_{\pi}\Big(\ldots\Big)=\vp \Big( X_{t_{i_1}}^{d_1} Y_{i_1} X_{t_{i_1}}^d Y_{i_1} X_{t_{i_1}}^{d_2} X_{t_{i_2}}^{d_2} X_{t_{i_2}}^{d_1} \Big)\vp\Big( X_{t_{i_2}}^d \Big) \vp\Big( Y_{i_2}^2\Big)$$
and
$$\sum_{\substack{\pi\in NC_2(D)\\(d_1+d+1,d_1+d+2) \in \pi\\(D-(d_1+d+1),D-d_1)\in \pi}} \vp_{\pi}\Big(\ldots\Big)=\vp \Big( X_{t_{i_1}}^{d_1} X_{t_{i_1}}^{d_2} X_{t_{i_2}}^{d_2} X_{t_{i_2}}^{d_1} \Big)\vp\Big(X_{t_{i_1}}^d\Big) \vp\Big(X_{t_{i_2}}^d\Big)\vp\Big( Y_{i_1}^2\Big) \vp\Big( Y_{i_2}^2\Big)\, .$$
Therefore, combining \eqref{exp-i-n} and \eqref{sum-ka-pi}, we end up with the formula
$$\big\|I_n\big\|_{L^2(\vp)}^2=\frac14\sum_{i_1,i_2=0}^{2^n-1}\sum_{\substack{\pi\in NC_2(D)\\(d_1+1,d_1+d+2) \notin \pi\\(D-(d_1+d+1),D-d_1)\notin \pi}} \vp_{\pi}\Big(\ldots\Big) \, ,$$where $\vp_\pi\big(\ldots \big)$ stands (again) for the quantity described in \eqref{ka-pi}.

\smallskip

Now, given $\pi\in NC_2(D)$ such that $(d_1+1,d_1+d+2) \notin \pi$ and $(D-(d_1+d+1),D-d_1)\notin \pi$, and considering $\vp_\pi\big(\ldots\big)$ in \eqref{ka-pi}, we know that the two occurrences of $Y_{i_1}$ cannot be connected via $\pi$, and the same is true for the two occurrences of $Y_{i_2}$, which easily leads us to the estimate 
\begin{align}
\big\|I_n\big\|_{L^2(\vp)}^2&\lesssim\sum_{a,b,c,d=1}^2\sum_{i_1,i_2=0}^{2^n-1} \Big\{ \big|\vp\big( Y_{i_1} Y_{i_2}\big)\big|^2+\big|\vp\big(Y_{i_1}Y_{i_2}\big)\vp\big(Y_{i_1}X_{t_{i_a}}\big) \vp\big(Y_{i_2}X_{t_{i_b}}\big)\big|\nonumber\\
&\hspace{4cm}+\big|\vp\big(Y_{i_1}X_{t_{i_a}}\big)\vp\big(Y_{i_1}X_{t_{i_b}}\big) \vp\big(Y_{i_2}X_{t_{i_c}}\big)  \vp\big(Y_{i_2}X_{t_{i_d}}\big)\big|\Big\} \, ,\label{bou-i-n-inter}
\end{align}
where the proportional constant in $\lesssim$ depends on $d,d_1,d_2$ and $\sup_{u\in [0,T]}\|X_u\|_{L^\infty(\vp)}$, but not on $n$.

\smallskip

As far as the first summand is concerned, we have, by \eqref{2d-var-r},
\begin{equation}\label{bou-1-sum-i-1-i-2}
\sum_{i_1,i_2=0}^{2^n-1}  \big|\vp\big( Y_{i_1} Y_{i_2}\big)\big|^2=\sum_{i_1,i_2=0}^{2^n-1} \bigg| R\begin{pmatrix} t_{i_1}^n & t_{i_1+1}^n\\ t_{i_2}^n& t_{i_2+1}^n \end{pmatrix}\bigg|^{\frac{1}{2H}}\big|\vp\big( Y_{i_1} Y_{i_2}\big)\big|^{2-\frac{1}{2H}}\lesssim 2^{-n\varepsilon}  \big\|R\big\|^{\frac{1}{2H}}_{\frac{1}{2H}-var;[0,T]^2} \, ,
\end{equation}
for some $\varepsilon >0$. Then, for any fixed $a,b\in \{1,2\}$,
\begin{align}
&\sum_{i_1,i_2=0}^{2^n-1} \big|\vp\big(Y_{i_1}Y_{i_2}\big)\vp\big(Y_{i_1}X_{t_{i_a}}\big)\vp\big(Y_{i_2}X_{t_{i_b}}\big)\big|\nonumber\\
&\leq \bigg( \sum_{i_1,i_2=0}^{2^n-1} \big|\vp\big(Y_{i_1}Y_{i_2}\big)\big|^2\bigg)^{1/2}\bigg(\sum_{i_1,i_2=0}^{2^n-1} \big|\vp\big(Y_{i_1}X_{t_{i_a}}\big)\big|^2 \big|\vp\big(Y_{i_2}X_{t_{i_b}}\big)\big|^2 \bigg)^{1/2}\nonumber\\
&\lesssim \bigg(2^{-n\varepsilon}  \big\|R\big\|^{\frac{1}{2H}}_{\frac{1}{2H}-var;[0,T]^2}\bigg)^{1/2} \bigg(\sum_{i_1,i_2=0}^{2^n-1} |t_{i_1+1}-t_{i_1}|^{4H} |t_{i_2+1}-t_{i_2}|^{4H} \bigg)^{1/2} \lesssim  2^{-n\varepsilon} \, ,\label{bou-2}
\end{align}
where we have used the bound \eqref{bou-1-sum-i-1-i-2} and the estimate \eqref{bou-r-h-2}.
Finally, for any fixed $a,b,c,d\in \{1,2\}$, and with the same arguments as above,
\begin{align}
&\sum_{i_1,i_2=0}^{2^n-1} \big|\vp\big(Y_{i_1}X_{t_{i_a}}\big)\vp\big(Y_{i_1}X_{t_{i_b}}\big) \vp\big(Y_{i_2}X_{t_{i_c}}\big)  \vp\big(Y_{i_2}X_{t_{i_d}}\big)\big|\nonumber\\
&\leq \bigg(\sum_{i_1,i_2=0}^{2^n-1} \big|\vp\big(Y_{i_1}X_{t_{i_a}}\big)\big|^2 \big|\vp\big(Y_{i_2}X_{t_{i_c}}\big)\big|^2 \bigg)^{1/2}\bigg(\sum_{i_1,i_2=0}^{2^n-1} \big|\vp\big(Y_{i_1}X_{t_{i_b}}\big)\big|^2 \big|\vp\big(Y_{i_2}X_{t_{i_d}}\big)\big|^2 \bigg)^{1/2}\nonumber\\
&\lesssim \sum_{i_1,i_2=0}^{2^n-1} |t_{i_1+1}-t_{i_1}|^{4H} |t_{i_2+1}-t_{i_2}|^{4H} \lesssim 2^{-n\varepsilon} \, .\label{bou-3}
\end{align}
Injecting \eqref{bou-1-sum-i-1-i-2}, \eqref{bou-2} and \eqref{bou-3} into \eqref{bou-i-n-inter}, we get the desired bound: for some $\varepsilon >0$,
\begin{equation}\label{boun-i-n}
\big\|I_n\big\|_{L^2(\vp)} \lesssim 2^{-n\varepsilon} \ .
\end{equation}

\

\subsection{Bound for $II_{N,n}$} 
First, it can be checked that the following simplication occurs: for every $m\geq n$,
$$\mathbb{X}^{2,m+1}_{t_i^n,t^n_{i+1}} \big[ X_{t^n_i}^d\big]-\mathbb{X}^{2,m}_{t_i^n,t^n_{i+1}} \big[ X_{t^n_i}^d\big]=\frac12 \sum_{j\in S_i^{m,n}} \Big[ Y^{(m)}_{2j}X_{t^n_i}^dY^{(m)}_{2j+1}-Y^{(m)}_{2j+1}X_{t^n_i}^d Y^{(m)}_{2j} \Big] \, ,$$
where $Y^{(m)}_j$ is the notation introduced in \eqref{y-m-j}, and $S_i^{m,n}:= \{j\geq 0: \ i\, 2^{m-n}\leq j\leq (i+1)2^{m-n}-1\}$. Consequently,
\begin{align*}
&\big\| II_{N,n} \big\|_{L^2(\vp)}^2\\
&=\frac14\sum_{i_1,i_2=0}^{2^n-1} \sum_{m_1,m_2=n}^{N-1} \sum_{j_1\in S_{i_1}^{m_1,n}} \sum_{j_2\in S_{i_2}^{m_2,n}} \vp\bigg( \Big\{ X_{t_{i_1}^n}^{d_1} \Big[Y^{(m_1)}_{2j_1} X^d_{t_{i_1}^n} Y^{(m_1)}_{2j_1+1}-Y^{(m_1)}_{2j_1+1}X^d_{t_{i_1}^n} Y^{(m_1)}_{2j_1} \Big] X_{t_{i_1}^n}^{d_2}\Big\}\\
&\hspace{6cm} \cdot \Big\{ X_{t_{i_2}^n}^{d_2} \Big[Y^{(m_2)}_{2j_2+1} X^d_{t_{i_2}^n} Y^{(m_2)}_{2j_2}-Y^{(m_2)}_{2j_2}X^d_{t_{i_2}^n} Y^{(m_2)}_{2j_2+1} \Big] X_{t_{i_2}^n}^{d_1}\Big\}\bigg)
\end{align*}
Using the NC Wick formula \eqref{form-wick}, we then get
\begin{equation}\label{bou-i-i-n-n-k}
\big\| II_{N,n} \big\|_{L^2(\vp)}^2=\frac14\sum_{i_1,i_2=0}^{2^n-1} \sum_{m_1,m_2=n}^{N-1} \sum_{j_1\in S_{i_1}^{m_1,n}} \sum_{j_2\in S_{i_2}^{m_2,n}} \sum_{\pi \in NC_2(D)}\Phi_\pi(i_1,i_2,m_1,m_2,j_1,j_2) \ ,
\end{equation}
where $D:=2(d+d_1+d_2)+4$ (just as in \eqref{ka-pi}) and
\begin{align*}
&\Phi_\pi(i_1,i_2,m_1,m_2,j_1,j_2)\\
&=\Big[ \vp_\pi^{i_1,i_2,m_1,m_2}(2j_1,2j_1+1,2j_2+1,2j_2)-\vp_\pi^{i_1,i_2,m_1,m_2}(2j_1,2j_1+1,2j_2,2j_2+1)\\
&\hspace{1cm}-\vp_\pi^{i_1,i_2,m_1,m_2}(2j_1+1,2j_1,2j_2+1,2j_2)+\vp_\pi^{i_1,i_2,m_1,m_2}(2j_1+1,2j_1,2j_2,2j_2+1) \Big] \ ,
\end{align*}
with
\begin{align*}
&\vp_\pi^{i_1,i_2,m_1,m_2}(k_1,\ell_1,k_2,\ell_2)\\
&:=\vp_\pi\Big(\underbrace{X_{t_{i_1}^n},\ldots,X_{t_{i_1}^n}}_{d_1 \, \text{occur.}},Y^{(m_1)}_{k_1},\underbrace{X_{t^n_{i_1}},\ldots,X_{t^n_{i_1}}}_{d \, \text{occur.}},Y^{(m_1)}_{\ell_1},\underbrace{X_{t^n_{i_1}},\ldots,X_{t^n_{i_1}}}_{d_2 \, \text{occur.}},\nonumber\\
&\hspace{4cm}\underbrace{X_{t^n_{i_2}},\ldots,X_{t^n_{i_2}}}_{d_2 \, \text{occur.}},Y^{(m_2)}_{k_2},\underbrace{X_{t^n_{i_2}},\ldots,X_{t^n_{i_2}}}_{d \, \text{occur.}},Y^{(m_2)}_{\ell_2},\underbrace{X_{t^n_{i_2}},\ldots,X_{t^n_{i_2}}}_{d_1 \, \text{occur.}} \Big) \, .
\end{align*}

At this point, the key observation is that for any $\pi\in NC_2(D)$ such that $(d_1+1,d_1+d+2)\in \pi$ or $(D-(d_1+d+1),D-d_1)\in \pi$, one has $\Phi_\pi(i_1,i_2,m_1,m_2,j_1,j_2)=0$ (just because $\vp(Y^{(m)}_{2j}Y^{(m)}_{2j+1})=\vp(Y^{(m)}_{2j+1}Y^{(m)}_{2j})$), and therefore we can restrict the sum over $NC_2(D)$ in \eqref{bou-i-i-n-n-k} to a sum over the set $\{\pi\in NC_2(D): \ (d_1+1,d_1+d+2)\notin \pi \ \text{and}\ (D-(d_1+d+1),D-d_1)\notin \pi\}$.

\smallskip

Based on this observation, we easily get that
\begin{align}
&\big\| II_{N,n} \big\|_{L^2(\vp)}^2\lesssim\sum_{i_1,i_2=0}^{2^n-1} \sum_{m_1,m_2=n}^{N-1} \sum_{j_1\in S_{i_1}^{m_1,n}} \sum_{j_2\in S_{i_2}^{m_2,n}}\sum_{a,b,c,d=1}^2\nonumber\\
&\hspace{1cm} \bigg\{ \Big[\big|\vp\big(Y^{(m_1)}_{2j_1} Y^{(m_2)}_{2j_2}\big) \big|\big|\vp\big(Y^{(m_1)}_{2j_1+1} Y^{(m_2)}_{2j_2+1}\big) \big|  +\big|\vp\big(Y^{(m_1)}_{2j_1} Y^{(m_2)}_{2j_2+1}\big) \big|\big|\vp\big(Y^{(m_1)}_{2j_1+1} Y^{(m_2)}_{2j_2}\big) \big|\Big]\label{line-1}\\
&\hspace{1.5cm}+\big|\vp\big(Y^{(m_1)}_{2j_1} X_{t_{i_a}^n}\big) \big| \big|\vp\big(Y^{(m_1)}_{2j_1+1} X_{t_{i_b}^n}\big) \big| \big|\vp\big(Y^{(m_2)}_{2j_2} X_{t_{i_c}^n}\big) \big| \big|\vp\big(Y^{(m_2)}_{2j_2+1} X_{t_{i_d}^n}\big) \big|\nonumber\\
&\hspace{1.5cm}+\Big[\big|\vp\big(Y^{(m_1)}_{2j_1} Y^{(m_2)}_{2j_2}\big) \big|\big|\vp\big(Y^{(m_1)}_{2j_1+1} X_{t_{i_a}^n}\big) \big| \big|\vp\big(Y^{(m_2)}_{2j_2+1} X_{t_{i_b}^n}\big) \big|\nonumber\\
&\hspace{2.5cm}+\big|\vp\big(Y^{(m_1)}_{2j_1} Y^{(m_2)}_{2j_2+1}\big) \big|\big|\vp\big(Y^{(m_1)}_{2j_1+1} X_{t_{i_a}^n}\big) \big| \big|\vp\big(Y^{(m_2)}_{2j_2} X_{t_{i_b}^n}\big) \big|\label{line-4}\\
&\hspace{2.5cm}+\big|\vp\big(Y^{(m_1)}_{2j_1+1} Y^{(m_2)}_{2j_2}\big) \big|\big|\vp\big(Y^{(m_1)}_{2j_1} X_{t_{i_a}^n}\big) \big| \big|\vp\big(Y^{(m_2)}_{2j_2+1} X_{t_{i_b}^n}\big) \big|\label{line-5}\\
&\hspace{2.5cm}+\big|\vp\big(Y^{(m_1)}_{2j_1+1} Y^{(m_2)}_{2j_2+1}\big) \big|\big|\vp\big(Y^{(m_1)}_{2j_1} X_{t_{i_a}^n}\big) \big| \big|\vp\big(Y^{(m_2)}_{2j_2} X_{t_{i_b}^n}\big) \big|\Big]\bigg\} \, .\label{line-6}
\end{align}

\

For all fixed $i_1,i_2,m_1,m_2$, it holds that
\begin{align*}
& \sum_{j_1\in S_{i_1}^{m_1,n}} \sum_{j_2\in S_{i_2}^{m_2,n}} \big|\vp\big(Y^{(m_1)}_{2j_1} Y^{(m_2)}_{2j_2}\big) \big|\big|\vp\big(Y^{(m_1)}_{2j_1+1} Y^{(m_2)}_{2j_2+1}\big) \big|\\
&\leq \bigg( \sum_{j_1\in S_{i_1}^{m_1,n}} \sum_{j_2\in S_{i_2}^{m_2,n}} \big|\vp\big(Y^{(m_1)}_{2j_1} Y^{(m_2)}_{2j_2}\big) \big|^2\bigg)^{1/2}\bigg( \sum_{j_1\in S_{i_1}^{m_1,n}} \sum_{j_2\in S_{i_2}^{m_2,n}} \big|\vp\big(Y^{(m_1)}_{2j_1+1} Y^{(m_2)}_{2j_2+1}\big) \big|^2\bigg)^{1/2} \, .
\end{align*}
Now
\begin{align*}
\sum_{j_1\in S_{i_1}^{m_1,n}} \sum_{j_2\in S_{i_2}^{m_2,n}} \big|\vp\big(Y^{(m_1)}_{2j_1} Y^{(m_2)}_{2j_2}\big) \big|^2
&\leq \sum_{j_1\in S_{i_1}^{m_1,n}} \sum_{j_2\in S_{i_2}^{m_2,n}}\bigg| R_H\begin{pmatrix}t_{2j_1}^{m_1} & t_{2j_1+1}^{m_1}\\ t_{2j_2}^{m_2} & t_{2j_2+1}^{m_2} \end{pmatrix}\bigg|^{\frac{1}{2H}}\big|\vp\big(Y^{(m_1)}_{2j_1} Y^{(m_2)}_{2j_2}\big) \big|^{2-\frac{1}{2H}}\\
&\lesssim 2^{-m_1 \varepsilon} 2^{-m_2 \varepsilon}\big\|R_H\big\|^{\frac{1}{2H}}_{\frac{1}{2H}-var;[t_{i_1}^n,t_{i_1+1}^n] \times [t_{i_2}^n,t_{i_2+1}^n]} \, ,
\end{align*}
and, with the same arguments,
\begin{align*}
&\sum_{j_1\in S_{i_1}^{m_1,n}} \sum_{j_2\in S_{i_2}^{m_2,n}} \big|\vp\big(Y^{(m_1)}_{2j_1+1} Y^{(m_2)}_{2j_2+1}\big) \big|^2\lesssim  2^{-m_1 \varepsilon} 2^{-m_2 \varepsilon}\big\|R_H\big\|^{\frac{1}{2H}}_{\frac{1}{2H}-var;[t_{i_1}^n,t_{i_1+1}^n] \times [t_{i_2}^n,t_{i_2+1}^n]} \ .
\end{align*}
Therefore, remembering that the map $\omega_H(I\times J):=\big\|R_H\big\|^{\frac{1}{2H}}_{\frac{1}{2H}-var;I\times J}$ defines a 2D-control, we get 
\begin{equation}\label{bound-pr-1}
\sum_{i_1,i_2=0}^{2^n-1} \sum_{m_1,m_2=n}^{N-1}\sum_{j_1\in S_{i_1}^{m_1,n}} \sum_{j_2\in S_{i_2}^{m_2,n}} \big|\vp\big(Y^{(m_1)}_{2j_1} Y^{(m_2)}_{2j_2}\big) \big|^2\lesssim 2^{-n\varepsilon} \big\|R_H\big\|^{\frac{1}{2H}}_{\frac{1}{2H}-var;[0,T]^2} \, .
\end{equation}
The very same arguments (and the very same resulting bound) clearly remain valid for the sum of the second term on line \eqref{line-1}.

\

Then, for any fixed $a,b,c,d\in \{1,2\}$, we have by \eqref{bou-r-h-2}
\begin{align*}
&\sum_{j_1\in S_{i_1}^{m_1,n}} \sum_{j_2\in S_{i_2}^{m_2,n}} \big|\vp\big(Y^{(m_1)}_{2j_1} X_{t_{i_a}^n}\big) \big| \big|\vp\big(Y^{(m_1)}_{2j_1+1} X_{t_{i_b}^n}\big) \big| \big|\vp\big(Y^{(m_2)}_{2j_2} X_{t_{i_c}^n}\big) \big| \big|\vp\big(Y^{(m_2)}_{2j_2+1} X_{t_{i_d}^n}\big) \big|\\
&\leq \bigg( \sum_{j_1\in S_{i_1}^{m_1,n}} \sum_{j_2\in S_{i_2}^{m_2,n}} \big|\vp\big(Y^{(m_1)}_{2j_1} X_{t_{i_a}^n}\big) \big|^2 \big|\vp\big(Y^{(m_2)}_{2j_2} X_{t_{i_c}^n}\big) \big|^2 \bigg)^{1/2}\\
&\hspace{4cm} \bigg(  \sum_{j_1\in S_{i_1}^{m_1,n}} \sum_{j_2\in S_{i_2}^{m_2,n}}  \big|\vp\big(Y^{(m_1)}_{2j_1+1} X_{t_{i_b}^n}\big) \big|^2  \big|\vp\big(Y^{(m_2)}_{2j_2+1} X_{t_{i_d}^n}\big) \big|^2 \bigg)^{1/2}\\
&\lesssim\sum_{j_1\in S_{i_1}^{m_1,n}} \sum_{j_2\in S_{i_2}^{m_2,n}} 2^{-4H m_1}2^{-4H m_2} \lesssim 2^{-2n} 2^{-m_1(4H-1)} 2^{-m_2(4H-1)} \ ,
\end{align*}
and thus
\begin{align}
&\sum_{i_1,i_2=0}^{2^n-1} \sum_{m_1,m_2=n}^{N-1}\sum_{j_1\in S_{i_1}^{m_1,n}} \sum_{j_2\in S_{i_2}^{m_2,n}} \big|\vp\big(Y^{(m_1)}_{2j_1} X_{t_{i_a}^n}\big) \big| \big|\vp\big(Y^{(m_1)}_{2j_1+1} X_{t_{i_b}^n}\big) \big| \big|\vp\big(Y^{(m_2)}_{2j_2} X_{t_{i_c}^n}\big) \big| \big|\vp\big(Y^{(m_2)}_{2j_2+1} X_{t_{i_d}^n}\big) \big|\nonumber\\
&\hspace{4cm}\lesssim \sum_{m_1,m_2=n}^{N-1} 2^{-m_1(4H-1)} 2^{-m_2(4H-1)} \lesssim 2^{-2n(4H-1)} \ .\label{bound-pr-2}
\end{align}

\

Finally, for any fixed $a,b\in \{1,2\}$, we can first combine the above intermediate bounds to get that
\begin{align*}
&\sum_{j_1\in S_{i_1}^{m_1,n}} \sum_{j_2\in S_{i_2}^{m_2,n}}\big|\vp\big(Y^{(m_1)}_{2j_1} Y^{(m_2)}_{2j_2}\big) \big|\big|\vp\big(Y^{(m_1)}_{2j_1+1} X_{t_{i_a}^n}\big) \big| \big|\vp\big(Y^{(m_2)}_{2j_2+1} X_{t_{i_b}^n}\big) \big|\\
&\leq \bigg( \sum_{j_1\in S_{i_1}^{m_1,n}} \sum_{j_2\in S_{i_2}^{m_2,n}}\big|\vp\big(Y^{(m_1)}_{2j_1} Y^{(m_2)}_{2j_2}\big) \big|^2\bigg)^{1/2}\\
&\hspace{4cm}\bigg( \sum_{j_1\in S_{i_1}^{m_1,n}} \sum_{j_2\in S_{i_2}^{m_2,n}}\big|\vp\big(Y^{(m_1)}_{2j_1+1} X_{t_{i_a}^n}\big) \big|^2 \big|\vp\big(Y^{(m_2)}_{2j_2+1} X_{t_{i_b}^n}\big) \big|^2\bigg)^{1/2}\\
&\lesssim 2^{-n} 2^{-m_1 \varepsilon} 2^{-m_2 \varepsilon}  \Big( \big\|R_H\big\|^{\frac{1}{2H}}_{\frac{1}{2H}-var;[t_{i_1}^n,t_{i_1+1}^n] \times [t_{i_2}^n,t_{i_2+1}^n]} \Big)^{1/2} \ ,
\end{align*}
for some $\varepsilon >0$, and as a result
\begin{align}
&\sum_{i_1,i_2=0}^{2^n-1} \sum_{m_1,m_2=n}^{N-1}\sum_{j_1\in S_{i_1}^{m_1,n}} \sum_{j_2\in S_{i_2}^{m_2,n}}\big|\vp\big(Y^{(m_1)}_{2j_1} Y^{(m_2)}_{2j_2}\big) \big|\big|\vp\big(Y^{(m_1)}_{2j_1+1} X_{t_{i_a}^n}\big) \big| \big|\vp\big(Y^{(m_2)}_{2j_2+1} X_{t_{i_b}^n}\big) \big|\nonumber\\
&\lesssim 2^{-n(1+2\varepsilon)} \sum_{i_1,i_2=0}^{2^n-1}\Big( \big\|R_H\big\|^{\frac{1}{2H}}_{\frac{1}{2H}-var;[t_{i_1}^n,t_{i_1+1}^n] \times [t_{i_2}^n,t_{i_2+1}^n]} \Big)^{1/2}\nonumber\\
&\lesssim 2^{-2n\varepsilon } \Big( \sum_{i_1,i_2=0}^{2^n-1}\big\|R_H\big\|^{\frac{1}{2H}}_{\frac{1}{2H}-var;[t_{i_1}^n,t_{i_1+1}^n] \times [t_{i_2}^n,t_{i_2+1}^n]} \Big)^{1/2}\lesssim 2^{-2n\varepsilon } \big\|R_H\big\|^{\frac{1}{4H}}_{\frac{1}{2H}-var;[0,T]^2}  \ .\label{bound-pr-3}
\end{align}
The very same arguments (and the very same resulting bound) clearly remain valid for the terms on lines \eqref{line-4}, \eqref{line-5} and \eqref{line-6}.

\

Combining \eqref{bound-pr-1}, \eqref{bound-pr-2} and \eqref{bound-pr-3}, we end up with the desired bound, namely
\begin{equation}\label{boun-ii-n-n}
\big\| II_{N,n} \big\|_{L^2(\vp)} \lesssim 2^{-n\varepsilon} \ ,
\end{equation}
where the proportional constant in $\lesssim$ depends neither on $n$, nor on $N$.

\

\subsection{Conclusion}
By injecting \eqref{boun-i-n} and \eqref{boun-ii-n-n} into the decomposition \eqref{decompo-i-ii}, we obtain that
$$\bigg\|\sum_{i=0}^{2^n-1} X_{t_i^n}^{d_1}\Big\{\mathbb{X}^{2,N}_{t_i^n,t^n_{i+1}} \big[ X_{t^n_i}^d\big] -\frac12 \vp\big( X_{t_i^n}^d\big) \vp\big( X_{t_i^n,t_{i+1}^n}^2\big)\Big\}X_{t_i^n}^{d_2} \bigg\|_{L^2(\vp)} \lesssim  2^{-n\varepsilon} \, ,$$
where the proportional constant in $\lesssim$ depends neither on $n$, nor on $N$. We can thus let $N$ tend to infinity (for fixed $n$) and use Proposition \ref{prop:rough-int} (point $(i)$) to assert that
$$\bigg\|\sum_{i=0}^{2^n-1} X_{t_i^n}^{d_1}\Big\{\mathbb{X}^{2}_{t_i^n,t^n_{i+1}} \big[ X_{t^n_i}^d\big] -\frac12 \vp\big( X_{t_i^n}^d\big) \vp\big( X_{t_i^n,t_{i+1}^n}^2\big)\Big\}X_{t_i^n}^{d_2} \bigg\|_{L^2(\vp)} \lesssim 2^{-n\varepsilon} \, ,$$
which immediately yields the convergence property \eqref{conv-loc-lev}.

\

\end{document}